\documentclass[final]{amsart}

\usepackage[english]{babel}
\usepackage[utf8]{inputenc}
\usepackage[colorinlistoftodos]{todonotes}
\usepackage{esint}
\usepackage{mathrsfs}
\usepackage{breqn}
\usepackage{graphics}
\usepackage{enumerate}
\usepackage{graphicx}
\usepackage{placeins}
\usepackage{amsmath,amsthm,amsfonts,amssymb,euscript,verbatim,bbm}
\usepackage[hidelinks]{hyperref}
\usepackage{amsrefs}

\usepackage{color}

\newtheorem{thm}{Theorem}
\newtheorem{remark}[thm]{Remark}
\newtheorem{prop}[thm]{Proposition}

\newtheorem{lemma}[thm]{Lemma}

\newcommand{\eqdef}{\overset{\mbox{\tiny{def}}}{=}}

\newcommand{\al}{\alpha}
\newcommand{\md}{\mathcal{D}}

\newcommand{\pone}{\partial_{\alpha_1}}
\newcommand{\ptwo}{\partial_{\alpha_2}}

\newcommand{\fzerone}{\mathcal{F}^{0,1}}
\newcommand{\foneone}{\dot{\mathcal{F}}^{1,1}}
\newcommand{\ftwoone}{\dot{\mathcal{F}}^{2,1}}

\newcommand{\yo}{\|f\|_{\dot{\mathcal{F}}^{1,1}}}

\newcommand{\ba}{\begin{equation}}
\newcommand{\ea}{\end{equation}}

\newcommand{\bea}{\begin{eqnarray}}
\newcommand{\eea}{\end{eqnarray}}

\def\beaa{\begin{eqnarray*}}
\def\eeaa{\end{eqnarray*}}
\begin{document}

\title[Muskat Problem with Viscosity Jump 3D]{ On the Muskat problem with viscosity jump: global in time results}

\author{F. Gancedo}

\author{E. Garc\'ia-Ju\'arez}

\author{N. Patel}

\author{R. M. Strain}

\date{\today; 
}

\begin{abstract}
The Muskat problem models the filtration of two incompressible immiscible fluids of different characteristics in porous media. In this paper, we consider both the 2D and 3D setting of two fluids of different constant densities and different constant viscosities. In this situation, the related contour equations are non-local, not only in the evolution system, but also in the implicit relation between the amplitude of the vorticity and the free interface. Among other extra difficulties, no maximum principles are available for the amplitude and the slopes of the interface in $L^\infty$. We prove global in time existence results for medium size  initial stable data in critical spaces. We also enhance previous methods by showing smoothing (instant analyticity), improving the medium size constant in 3D, together with sharp decay rates of analytic norms. The found technique is twofold, giving ill-posedness in unstable situations for very low regular solutions.   
\end{abstract}

\setcounter{tocdepth}{1}

\maketitle
\tableofcontents

\section{Introduction}

This paper studies the dynamics of flows in porous media. This scenario is modeled using the classical Darcy's law \cite{Darcy56}
\begin{equation}\label{Darcy}
\mu(x,t) u(x,t)=-\nabla p(x,t)-\rho(x,t)e_d, 
\end{equation} 
where the velocity of the fluid $u$ is proportional to the spatial gradient pressure $\nabla p$ and the gravity force. Above $x$ is the space variable in $\mathbb{R}^d$ for $d=2,$ or $3$, $t\geq 0$ is time and $e_d$ is the last canonical vector. In the momentum equation, velocity replaces flow acceleration due to the porosity of the medium. It appears with the viscosity $\mu(x,t)$ divided by the permeability constant $\kappa$, here equal to one for simplicity of the exposition. The gravitational field comes with the density of the fluid $\rho(x,t)$ multiplied by the gravitational constant $g$, which is also normalized to one for clarity.

In this work the flow is incompressible 
\begin{equation}\label{incom}
\nabla \cdot u(x,t)=0,
\end{equation}
and takes into consideration the dynamics of two immiscible fluids permeating the porous medium $\mathbb{R}^d$ with different constant densities and viscosities 
\begin{equation}\label{patchsolution}
\mu(x,t)=\left\{\begin{array}{rl}
\mu^1,& x\in D^1(t),\\
\mu^2,& x\in D^2(t),
\end{array}\right.
\quad 
\rho(x,t)=\left\{\begin{array}{rl}
\rho^1,& x\in D^1(t),\\
\rho^2,& x\in D^2(t).
\end{array}\right.
\end{equation}
The open sets $D^1(t)$ and $D^2(t)$ are connected with $\mathbb{R}^d=D^1(t)\cup D^2(t)\cup \partial D^j(t)$, $j=1$, 2 and move with the velocity of the fluid
\begin{equation}\label{dynamics}
\frac{dx}{dt}(t)=u(x(t),t),\quad \forall\, x(t)\in D^j(t), \mbox{ or } x(t)\in \partial D^j(t).
\end{equation}
The evolution equation above is well-defined at the free boundary even though the velocity is not continuous. The discontinuity in the velocity holds due to the density and viscosity jumps. But what matters is the velocity in the normal direction, which is continuous thanks to the incompressibility of the velocity. The geometry of the problem is due to the gravitational force, with the fluid of viscosity $\mu^2$ and density $\rho^2$ located mainly below the fluid of viscosity $\mu^1$ and density $\rho^1$. In particular, there exists a constant $M>1$ large enough such that $\mathbb{R}^{d-1}\times (-\infty,-M]\subset D^2(t)$. 

We are then dealing with the well-established Muskat problem, whose main interest is about the dynamics of the free boundary $\partial D^j(t)$, especially between water and oil \cite{Muskat34}. In this paper, we study precisely this density-viscosity jump  scenario, i.e. when there is a viscosity jump together with a density jump between the two fluids. Due to its wide applicability, this problem has been extensively studied \cite{Bear72}. In particular from the physical and experimental point of view, as in the two-dimensional case the phenomena is mathematically analogous to the two-phase Hele-Shaw cell evolution problem \cite{SaffmanTaylor58}.

From the mathematical point of view, in the last decades there has been a lot of effort to understand the problem  as it generates very interesting incompressible fluid dynamics behavior \cite{Gan17}. 

An important characteristic of the problem is that its Eulerian-Lagrangian formulation (\ref{Darcy},\ref{incom},\ref{patchsolution},\ref{dynamics}) understood in a weak sense provides an equivalent self-evolution equation for the interface $\partial D^j(t)$. This is the so-called contour evolution system, which we will now provide for 3D Muskat in order to understand the dynamics of its solutions. 

Due to the irrotationality of the velocity in each domain $D^j(t)$, the vorticity is concentrated on the interface $\partial D^j(t)$. That is the vorticity is given by a delta distribution as follows
$$
\nabla\wedge u(x,t)=\omega(\alpha,t)\delta(x=X(\alpha,t)),
$$
where $\omega(\alpha,t)$ is the amplitude of the vorticity and $X(\alpha,t)$ is a global parameterization of $\partial D^j(t)$ with
$$
\partial D^j(t)=\{X(\alpha,t):\, \alpha\in\mathbb{R}^2\}.
$$
It means that
$$
\int_{\mathbb{R}^3}u(x,t)\cdot\nabla\wedge\varphi(x)dx=\int_{\mathbb{R}^2}\omega(\alpha,t)\cdot\varphi(X(\alpha,t))d\alpha,
$$
for any smooth compactly supported field $\varphi$. The evolution equation reads
\begin{equation}\label{eqevcon}
\partial_tX(\alpha,t)=BR(X,\omega)(\alpha,t)+C_1(\alpha,t)\partial_{\alpha_{1}} X(\alpha,t)+C_2(\alpha,t)\partial_{\alpha_{2}} X(\alpha,t),
\end{equation}
where $BR$ is the well-known Birkhoff-Rott integral
\begin{equation}\label{eqBR}
BR(X,\omega)(\alpha,t)=-\frac{1}{4\pi}\text{p.v.}\int_{\mathbb{R}^2}\frac{X(\alpha,t)-X(\beta,t)}{|X(\alpha,t)-X(\beta,t)|^3}\wedge \omega(\beta,t) d\beta,
\end{equation}
which appears using the Biot-Savart law and taking limits to the free boundary. 
Above the coefficients $C_1$ and $C_2$ represent the possible change of coordinates for the evolving surface and the prefix p.v. indicates a principal value integral. It is possible to close the system using that the velocity is given by different potentials in each domain and we denote the potential jump across the interface by the  function $\Omega(\alpha,t)$. Taking limits approaching the free boundary in Darcy's law yields the non-local implicit identity
\begin{equation}\label{eqOmega}
\Omega(\alpha,t)=A_\mu\mathcal{D}(\Omega)(\alpha,t)-2A_\rho X_3(\alpha,t),\quad A_\mu=\frac{\mu^2-\mu^1}{\mu^2+\mu^1}, \quad A_\rho=\frac{\rho^2-\rho^1}{\mu^2+\mu^1},
\end{equation}
where $\mathcal{D}$ is the double layer potential
\begin{equation}\label{dlp}
\mathcal{D}(\Omega)(\alpha,t)=
\frac{1}{2\pi}\text{p.v.}\!\!\int_{\mathbb{R}^2}\! \frac{X(\alpha,t)\!-\!X(\beta,t)}{|X(\alpha,t)\!-\!X(\beta,t)|^3}\cdot \partial_{\alpha_{1}}X(\beta,t)\wedge\partial_{\alpha_{2}}X(\beta,t)\Omega(\beta,t) d\beta.
\end{equation}
In that limit procedure, the continuity of the pressure at the free boundary is used, which is a consequence of the fact that Darcy's law \eqref{Darcy} is understood in a weak sense. Relating the potential and the velocity jumps at the interface provides
\begin{equation}\label{omo}
\omega(\alpha,t)=\partial_{\alpha_{2}}\Omega(\alpha,t) \partial_{\alpha_{1}} X(\alpha,t)-\partial_{\alpha_{1}}\Omega(\alpha,t) \partial_{\alpha_{2}} X(\alpha,t),
\end{equation}
and therefore it is possible to close the contour evolution system by (\ref{eqevcon},\ref{eqBR},\ref{eqOmega},\ref{dlp},\ref{omo}) (see \cite{CCG13} for a detail derivation of the system).

Then the next question to ask is about the well-posedness of the problem. A remarkable peculiarity is that, in general, it does not hold. The system has to initially satisfy the so-called Rayleigh-Taylor condition (also called the Saffman-Taylor condition for the Muskat problem) to be well-posed. This condition holds if the difference of the gradient of the pressure in the normal direction at the interface is strictly positive \cite{Ambrose04},\cite{Ambrose07}, i.e the stable regime. For large initial data, well-posedness was proved in \cite{CG07} for the case with density jump in two and three dimensions. In that case, the Saffman-Taylor condition holds if the denser fluid lies below the less dense fluid. The density-viscosity jump stable situation was proved to exist locally in time in 2D \cite{CCG11} and in 3D \cite{CCG13}. Although these proofs use different approaches, it was essential in both proofs to find bounds for the amplitude of the vorticity in equation \eqref{eqOmega} in terms of the free boundary. There are recent results where local-in-time existence is shown in 2D for lower regular initial data given by graphs in the Sobolev space $H^2$ for the one-fluid case ($\mu^2=0$) \cite{CGS16} and in the two-fluid case ($\mu^2\geq 0$) \cite{Mat17}. In the 2D density jump case the local-in-time existence has been shown for any subcritical Sobolev spaces $W^{2,p}$, $1<p<\infty$ \cite{CGSV17}, and $H^s$, $3/2<s<2$ \cite{Mat16}. Here, the terminology subcritical is used in terms of the scaling of the problem, as $X^\lambda(\alpha,t)=\lambda^{-1} X(\lambda\alpha,\lambda t)$ and $\omega^\lambda(\alpha,t)=\omega(\lambda\alpha,\lambda t)$ are solution of (\ref{eqevcon},\ref{eqBR},\ref{eqOmega},\ref{dlp},\ref{omo}) for any $\lambda\geq 0$ if $X(\alpha,t)$ and $\omega(\alpha,t)$ are. Therefore $\dot{W}^{1,\infty}$, $\dot{W}^{2,1}$ and $\dot{H}^{3/2}$ are critical and invariant homogeneous spaces for the system in 2D, or $\dot{W}^{1,\infty}$, $\dot{W}^{3,1}$ and $\dot{H}^{2}$ in the 3D case. It is then easy to check that the main space $\dot{\mathcal{F}}^{1,1}$ used in this paper (see definition below) is also critical. 

On the other hand, the Muskat problem can be unstable for some scenarios, when the Saffman-Taylor condition does not hold. In particular, if the difference of the gradient of the pressures in the normal direction at the interface is strictly negative, the contour evolution problem is ill-posed in the viscosity jump case \cite{SCH04} as well as the density jump situation \cite{CG07} in subcritical spaces. With the Eulerian-Lagrangian formulation (\ref{Darcy},\ref{incom},\ref{patchsolution},\ref{dynamics}) it is possible to find weak solutions in the density jump case where the fluid densities mix in a strip close to the flat steady unstable state \cite{Sze12} and for any $H^5$ unstable graph \cite{CCF16}. In the contour dynamics setting, adding capillary forces to the system makes the contour equation well-posed \cite{ES97}. When the Saffman-Taylor condition holds, the system is structurally stable to solutions without capillary forces if the surface tension coefficient is close to zero \cite{Ambrose14}. However, there exist unstable scenarios for interfaces interacting with capillary forces \cite{Ott97} which have been shown to have exponential growth for low order norms under small scales of times \cite{GHS07}. The system also exhibits finger shaped unstable stationary-states solutions \cite{EM11}.

A very important feature of this problem is that it can develop finite time singularities starting from stable situations. The Muskat problem then became the first incompressible model where blow-up for solutions with initial data in well-posed scenarios had been proven rigorously. Specifically, in the 2D density jump case, solutions starting in stable situation (denser fluid below a graph) become instantly analytic and move to unstable regimes in finite time \cite{CCFGL12}. In the unstable regime the interface is not a graph anymore, and at a later time the Muskat solution blows-up in finite time showing loss of regularity \cite{CCFG13}. The geometry of those initial data are not well understood, as numerical experiments show that some solutions with large initial data can remain smooth \cite{CGO08}, and the patterns can become more complicated for scenarios with fixed boundary effects \cite{GG14}. As a matter of fact, some solutions can pass from the stable to the unstable regime and enter again to the stable regime \cite{CSZ17}.

The Muskat problem also develops a different kind of blow-up behavior in stable regimes: the so-called splash singularities. This singularity occurs if two different particles on the free boundary collide in finite time while the regularity of the interface is preserved. This collision can not occur along a connected segment of the curve of particles in either the density jump \cite{CG10} or the density-viscosity jump case \cite{CP17}. In particular, the splash singularity is ruled out for the two-fluid case \cite{GS14} but it takes place in one-fluid stable scenarios \cite{CCFG16}.

The question we study in this paper is about the global in time existence, uniqueness, regularity and decay of solution of the Muskat problem in stable regimes and ill-posedness in unstable regimes. In the viscosity \cite{SCH04}, density \cite{CG07} and density-viscosity jump 2D cases \cite{EM11},
\cite{CGS16} there exist global in time classical solutions for small initial data in subcritical norms which become instantly analytic, thereby demonstrating the parabolic character of the system in these situations. See \cite{BSW14} for the same result in the 2D density jump case with small initial data in critical norms, represented on the Fourier side by positive measure. In \cite{CGSV17}, global in time existence of classical solutions are shown to exist with small initial slope. In \cite{CCGS13}, global existence of 2D density-jump Muskat Lipschitz solutions are given for initial data with slope less than one. See \cite{Gra14} for an extension of the result with fixed boundary and \cite{CCGRS16} for the 3D scenario, where the $L^\infty$ norm of the free boundary gradient has to be smaller than $1/3$. In \cite{CCGS13} and \cite{CCGRS16} global existence and uniqueness is proved for solutions with continuous and bounded slope and $L^1$ in time bounded curvature in the density jump case for initial data in critical spaces with medium size. More specifically, the initial profiles are given by functions, i.e. $X(\alpha,0) = (\alpha, f_{0}(\alpha))$, for a function $f_{0}(\alpha)$ of size less than $k_0$:
$$
\|f_0\|_{\dot{\mathcal{F}}^{1,1}}=\int_{\mathbb{R}^{d-1}}\! \!\! d\xi |\xi||\hat{f}_0(\xi)|< k_{0,d},\qquad d=2,3,
$$    
where $k_{0,d}$ is an explicit constant, $k_{0,3}> 1/5$ in 3D and $k_{0,2}>1/3$ in 2D. In \cite{PS17}, the optimal time decay of those solutions are proven, for initial data additionally bounded in subcritical Sobolev norms. We also point out work \cite{Cam17}, where the Lipschitz solutions given in \cite{CCFG13} are shown to become smooth by using a conditional regularity result given in \cite{CGSV17} together with an instant generation of a modulus of continuity.

Next, we describe the main results and novelties in this work. This paper extends the global existence results in 2D and 3D from \cite{CCGRS16} to the more general case with density-viscosity jump. Moreover, in 3D we improve the available constant for global existence and make it equal to the 2D constant in the $A_\mu=0$ case. Precisely, it is given by initial data satisfying that
$$
\|f_0\|_{\dot{\mathcal{F}}^{1,1}}=\int d\xi |\xi||\hat{f}_0(\xi)|< k(|A_{\mu}|),
$$ 
where $k:[0,1]\to [k(1),k_0]$ is decreasing and $k(0)=k_0=k_{0,2}$. We would like to point out that due to the nature of equation \eqref{eqOmega}, maximum principles are not available for the amplitude and the slopes in the $L^\infty$ norm and the parabolic character of the equation is not as clear as in the case $A_\mu=0$. We provide the first global existence result for this important scenario in a critical space. The space $\dot{\mathcal{F}}^{1,1}$ appears as a natural framework to perform the task of inverting the operator $(I-A_\mu\mathcal{D})$ in order to get bounds for $\omega$ in terms of the interface. In particular, we also improve the method in \cite{CCGRS16} as we are able to show smoothing effects, proving that solutions with medium size initial data become instantly analytic. Furthermore, we show uniform bounds of the interface in $L^{\infty}$ and $L^{2}$ norms with analytic weights. Then, we show optimal decay rates for the analyticity of the critical solutions, improving the results in \cite{PS17}.

Finally, we show with the new approach in the paper that solutions are ill-posed in the unstable regime even for low regularity solutions understood in the contour dynamics setting. We give precise statements of these results in Section \ref{MainResults}. In next section we provide the contour equations we use throughout the paper.

\section{Formulation of the Muskat Problem with Viscosity Jump}\label{formulation}
In this section, we derive the contour equation formula given by (\ref{eqevcon},\ref{eqBR},\ref{eqOmega},\ref{dlp},\ref{omo}) in terms of a graph. This equation will be used throughout the paper to state the main results and to prove them. To simplify notation we shall write $f(\alpha,t)=f(\al)$ when there is no danger of confusion. 

In the 3D case, if the evolving interface can be described as a graph
\begin{equation*}
X(\alpha,t)=(\alpha_1,\alpha_2,f(\alpha,t)),\hspace{0.6cm} \alpha=(\alpha_1,\alpha_2)\in\mathbb{R}^2,
\end{equation*}
then the equations \eqref{eqevcon} are reduced to one equation as follows
\begin{align*}
0&=-\frac{1}{4\pi}\textup{p.v.} \int_{\mathbb{R}^2} \frac{(\alpha_2-\beta_2)\omega_3(\beta)-\omega_2(\beta)(f(\alpha)-f(\beta))}{|(\alpha,f(\alpha))-(\beta,f(\beta))|^3}d\beta +C_1(\alpha),
\end{align*}

\begin{align*}
0&=-\frac{1}{4\pi}\textup{p.v.} \int_{\mathbb{R}^2} \frac{\omega_1(\beta)(f(\alpha)-f(\beta))-(\alpha_1-\beta_1)\omega_3(\beta)}{|(\alpha,f(\alpha))-(\beta,f(\beta))|^3}d\beta +C_2(\alpha),
\end{align*}
\begin{align*}
f_t(\alpha)&\!=\!-\frac{1}{4\pi}\!\int_{\mathbb{R}^2}\!\!\frac{(\alpha_1\!-\!\beta_1)\omega_2(\beta)\!-\!(\alpha_2\!-\!\beta_2)\omega_1(\beta)}{|(\alpha,f(\alpha))\!-\!(\beta,f(\beta))|^3}d\beta\!+\!C_1(\alpha)\partial_{\alpha_1}f(\alpha)\!+\!C_2(\alpha)\partial_{\alpha_2}f(\alpha).
\end{align*}
Thus, substituting the coefficients from the tangent terms into the evolution equation and applying a change of variables, we obtain the equation for $f$: 
\begin{equation}\label{ftI3}
f_{t}(\alpha) = I_{1}(\al) + I_{2}(\al) + I_{3}(\al),
\end{equation}
where
\begin{equation}
\label{I1}
I_{1}(\al) = -\frac{1}{4\pi}\text{p.v.} \int_{\mathbb{R}^{2}} \frac{\beta_{1} \omega_{2}(\alpha-\beta)-\beta_{2}\omega_{1}(\alpha-\beta)}{(1+(\Delta_{\beta}f(\alpha))^{2})^{\frac{3}{2}}}\frac{d\beta}{|\beta|^{3}} ,
\end{equation}
\begin{equation}
\label{I2}
I_{2}(\al)\! =\!\frac{1}{4\pi}\text{p.v.}\!\! \int_{\mathbb{R}^{2}} \!\!\frac{\Delta_{\beta}f(\alpha)\partial_{\alpha_{2}}f(\alpha) \omega_{1}(\alpha\!-\!\beta)\!-\!\Delta_{\beta}f(\alpha)\partial_{\alpha_{1}}f(\alpha) \omega_{2}(\alpha\!-\!\beta)}{(1+(\Delta_{\beta}f(\alpha))^{2})^{\frac{3}{2}}}\frac{d\beta}{|\beta|^{2}},
\end{equation}
\begin{equation}\label{I3}
I_{3}(\al)= \frac{1}{4\pi}\text{p.v.} \int_{\mathbb{R}^{2}} \frac{\beta_{2} \partial_{\alpha_{1}}f(\alpha)-\beta_{1}\partial_{\alpha_{2}}f(\alpha)}{(1+(\Delta_{\beta}f(\alpha))^{2})^{\frac{3}{2}}}\omega_{3}(\alpha-\beta)\frac{d\beta}{|\beta|^{3}}.
\end{equation}
Above we use the notation $\Delta_{\beta}f(\alpha)$ for
$$
\Delta_{\beta}f(\alpha)=(f(\al)-f(\al-\beta))|\beta|^{-1}.
$$
We have the following equations for the vorticity coming from \eqref{omo}:
\begin{align}\label{omegai}
\begin{split}
\omega_{1} = \partial_{\alpha_{2}}\Omega,
\quad
\omega_{2} = -\partial_{\alpha_{1}}\Omega,
\quad
\omega_{3} = \partial_{\alpha_{2}}\Omega\partial_{\alpha_{1}}f-\partial_{\alpha_{1}}\Omega\partial_{\alpha_{2}}f.
\end{split}
\end{align}

Introducing \eqref{omegai} into \eqref{I1} and \eqref{I2} they can be written as
\begin{equation*}
\begin{aligned}
I_1(\al)&= \frac{1}{4\pi}\text{p.v.} \int_{\mathbb{R}^{2}} \frac{1 }{(1+(\Delta_{\beta}f(\alpha))^{2})^{\frac{3}{2}}}\frac{\beta}{|\beta|^3}\cdot \nabla \Omega(\alpha-\beta) d\beta,
\end{aligned}
\end{equation*}
\begin{equation}\label{I2Amu}
\begin{aligned}
I_2(\al)&= \frac{1}{4\pi}\text{p.v.} \int_{\mathbb{R}^{2}} \frac{\Delta_\beta f(\alpha)\nabla f(\alpha)}{(1+(\Delta_{\beta}f(\alpha))^{2})^{\frac{3}{2}}}\cdot\frac{\nabla \Omega(\alpha-\beta)}{|\beta|^2} d\beta.
\end{aligned}
\end{equation}
By adding and subtracting the appropriate quantity, 
we obtain the following
\begin{equation*}
I_{1}(\al) =\frac12\Lambda \Omega(\al) + \frac{1}{4\pi}\textup{p.v.}\int_{\mathbb{R}^2} \Big((1+(\Delta_{\beta}f(\alpha))^{2})^{-\frac{3}{2}}-1\Big)\frac{\beta}{|\beta|^3}\cdot\nabla\Omega(\alpha-\beta)d\beta,
\end{equation*}
where the operator $\Lambda$ is given by the Riesz transforms 
\begin{equation}\label{riesz}
\Lambda=R_1\partial_{\alpha_1}+R_2\partial_{\alpha_2}
\end{equation} and also as a Fourier multiplier by $\widehat{\Lambda}=|\xi|$. Plugging the identity for $\Omega$ \eqref{eqOmega}, the equation below shows the parabolic structure of the equation as 
\begin{equation}\label{I1Amu}
\begin{aligned}
I_{1}(\al)=-&A_{\rho}\Lambda f(\alpha)+\frac{A_\mu}{2}\Lambda \md(\Omega)(\alpha)\\
&\quad+\frac{1}{4\pi}\textup{p.v.}\int_{\mathbb{R}^2} \Big((1+(\Delta_{\beta}f(\alpha))^{2})^{-\frac{3}{2}}-1\Big)\frac{\beta}{|\beta|^3}\cdot\nabla\Omega(\alpha-\beta)d\beta.
\end{aligned}	
\end{equation}
Using formulas \eqref{omegai} and \eqref{eqOmega} in $I_3(\al)$ \eqref{I3} we are able to find that
\begin{equation}\label{I3Amu}
I_{3}(\al)= \frac{A_\mu}{4\pi}\text{p.v.} \int_{\mathbb{R}^{2}} \frac{\beta \cdot\nabla^{\bot} f(\alpha) \nabla\md(\Omega
	)(\alpha-\beta)\cdot\nabla^{\bot}f(\alpha-\beta)}{(1+(\Delta_{\beta}f(\alpha))^{2})^{\frac{3}{2}}}\frac{d\beta}{|\beta|^{3}}.
\end{equation}
We can finally write the contour equation \eqref{ftI3} by using formulas \eqref{I1Amu}, \eqref{I2Amu} and \eqref{I3Amu} to get:
\begin{equation}\label{decompnonlinear}
f_{t} = -A_{\rho}\Lambda f+N(f),\qquad\mbox{where}\qquad N(f)=N_1(f)+N_2(f)+N_3(f),
\end{equation}
where $N(f) = N(f,\Omega)$ and
\begin{equation}\label{N1N2N3}
\begin{aligned}
N_1&= \frac{A_\mu}{2}\Lambda \mathcal{D}(\Omega)(\alpha),\\
N_{2}& = \frac{1}{4\pi}\text{p.v.}\int_{\mathbb{R}^2} \left( \frac{\frac{\beta}{|\beta|}+\Delta_{\beta}f(\alpha)\nabla f(\alpha)}{(1+(\Delta_{\beta}f(\alpha))^{2})^{3/2}} - \frac{\beta}{|\beta|} \right) \cdot\frac{\nabla\Omega(\alpha-\beta)}{|\beta|^{2}} d\beta,\\
N_3&=\frac{A_\mu}{4\pi}\text{p.v.} \int_{\mathbb{R}^{2}} \frac{\beta \cdot\nabla^{\bot} f(\alpha)\nabla\md(\Omega
	)(\alpha-\beta)\cdot\nabla^{\bot}f(\alpha-\beta)}{(1+(\Delta_{\beta}f(\alpha))^{2})^{\frac{3}{2}}}\frac{d\beta}{|\beta|^{3}}.
\end{aligned}
\end{equation}
The equation for $\Omega$ is given implicitly by
\begin{equation}\label{Omegagraph}
\Omega(\alpha,t)=A_\mu \md(\Omega)(\alpha,t)-2A_\rho f(\alpha,t),
\end{equation}
where the operator $\md(\Omega)$ \eqref{dlp} is rewritten as follows
\begin{equation}\label{DOmegagraph}
\md(\Omega)(\alpha) = \frac{1}{2\pi} \int_{\mathbb{R}^{2}} \frac{\Delta_{\beta}f(\alpha)-\frac{\beta\cdot \nabla f(\alpha-\beta)}{|\beta|}  }{(1+(\Delta_{\beta}f(\alpha))^{2})^{3/2}}\frac{\Omega(\alpha-\beta)}{|\beta|^{2}} d\beta.
\end{equation}
Note that the derivatives of $\md(\Omega)$ can be written in the following manner
\begin{equation}\label{partialD}
\partial_{\alpha_{i}}\md(\Omega)(\alpha,t) = 2 BR(f,\omega)(\alpha,t)\cdot \partial_{\alpha_{i}}(\alpha_{1},\alpha_{2},f(\alpha)),
\end{equation}
and therefore
\begin{equation}
\label{partialalphaOmega}
\partial_{\alpha_{i}}\Omega(\alpha,t) - 2A_{\mu} BR(f,\omega)(\alpha,t)\cdot \partial_{\alpha_{i}}(\alpha_{1},\alpha_{2},f(\alpha)) = -2A_{\rho}\partial_{\alpha_{i}}f(\alpha,t).
\end{equation}
In the case of a graph, the Birkhoff-Rott integrals are also reduced in the following manner
\begin{equation*}
BR(f,\omega) \eqdef (BR_{1}(f,\omega),BR_{2}(f,\omega), BR_{3}(f,\omega)),
\end{equation*}
where we use the shorthand $BR_{i} \eqdef BR_{i}(f,\omega)$ to be the terms
\begin{equation}
\label{br1}
BR_{1} = \frac{-1}{4\pi}\text{p.v.}\int_{\mathbb{R}^{2}}\frac{\frac{\beta_{2}}{|\beta|}\omega_{3}(\alpha-\beta) - \Delta_{\beta}f(\alpha)\omega_{2}(\alpha-\beta)}{(1+\Delta_{\beta}f(\alpha)^{2})^{\frac{3}{2}}} \frac{d\beta}{|\beta|^{2}},
\end{equation}
\begin{equation}
\label{br2}
BR_{2} = \frac{-1}{4\pi}\text{p.v.}\int_{\mathbb{R}^{2}}\frac{\Delta_{\beta}f(\alpha)\omega_{1}(\alpha-\beta)-\frac{\beta_{1}}{|\beta|}\omega_{3}(\alpha-\beta)}{(1+\Delta_{\beta}f(\alpha)^{2})^{\frac{3}{2}}}\frac{d\beta}{|\beta|^{2}},
\end{equation}
\begin{equation}
\label{br3}
BR_{3} = \frac{-1}{4\pi}\text{p.v.}\int_{\mathbb{R}^{2}}\frac{\frac{\beta_{1}}{|\beta|}\omega_{2}(\alpha-\beta)-\frac{\beta_{2}}{|\beta|}\omega_{1}(\alpha-\beta)}{(1+\Delta_{\beta}f(\alpha)^{2})^{\frac{3}{2}}}\frac{d\beta}{|\beta|^{2}}.
\end{equation}
This completes our explanation of the formulation in 3D.

We now state the formulation in 2D. Proceeding similarly as above one obtains that
\begin{equation}\label{graph2d}
f_{t} = -A_{\rho}\Lambda f+N(f),\qquad\mbox{where}\qquad N(f)=N_1(f)+N_2(f),
\end{equation}
where $N(f) = N(f,\Omega)$ and
\begin{equation}\label{nonlinear2d}
\begin{aligned}
N_1&= \frac{A_\mu}{2} \Lambda\mathcal{D}(\Omega)(\alpha),\\
N_2& = \frac{1}{2\pi}\text{p.v.}\int \frac{\Delta_\beta f(\alpha)-\partial_\alpha f(\alpha)}{1+\Delta_\beta f(\alpha)^2}\Delta_{\beta}f(\alpha)\frac{\partial_\alpha\Omega(\alpha-\beta)}{\beta} d\beta.
\end{aligned}
	\end{equation}
The equation for $\Omega$ is given implicitly by
\begin{equation}\label{omega2d}
	\Omega(\alpha,t)=A_\mu \mathcal{D}(\Omega)(\alpha,t)-2A_\rho  f(\alpha,t),
\end{equation}
where the operator $\mathcal{D}(\Omega)(\alpha,t)$ is rewritten as follows
\begin{equation}\label{t2d}
	\mathcal{D}(\Omega)(\alpha,t) = \frac{1}{\pi}\int_{\mathbb{R}} \frac{\Delta_\beta f(\alpha)-\partial_\alpha f(\alpha-\beta)}{1+\Delta_\beta f(\alpha)^2}\frac{\Omega(\alpha-\beta)}{\beta}d\beta.
\end{equation}
Note that the vorticity is given by $\omega(\alpha)=\partial_\alpha \Omega (\alpha)$.

\section{Main Results}\label{MainResults}

In this section, we present the main results and briefly give an outline of the structure of this paper. The first result is global well-posedness in the critical space $\mathcal{\dot{F}}^{1,1}\cap L^{2}$ in 3D, where we define the norms
$$
\|f\|_{\mathcal{\dot{F}}^{s,1}} \eqdef \int |\xi|^{s}|\hat{f}(\xi)| d\xi, \quad s> -2.
$$
We also denote $\|f\|_{\mathcal{\dot{F}}^{0,1}}=\|f\|_{\mathcal{F}^{0,1}}$.

\begin{thm}[Existence and Uniqueness, 3D]\label{Global}
	Let $f_{0}\in \dot{\mathcal{F}}^{1,1}\cap L^{2}$ satisfy the bound
	$$\|f_{0}\|_{\mathcal{\dot{F}}^{1,1}} < k(|A_{\mu}|)$$
	for a constant $k(|A_{\mu}|)$ depending on the Atwood number $A_{\mu}$. Then there exists a unique solution to (\ref{decompnonlinear}-\ref{DOmegagraph}) with $f\in L^{\infty}(0,T;\mathcal{\dot{F}}^{1,1}\cap L^{2})\cap L^{1}(0,T;\dot{\mathcal{F}}^{2,1})$ such that $f(\alpha,0) = f_{0}(\alpha)$ and
	\begin{equation}\label{globalineq}
	\|f\|_{L^2}(t)\leq \|f_0\|_{L^2},\quad \|f\|_{\mathcal{\dot{F}}^{1,1}}(t) + \sigma \int_{0}^{t}\|f\|_{\mathcal{\dot{F}}^{2,1}}(\tau)d\tau\leq \|f_{0}\|_{\mathcal{\dot{F}}^{1,1}} < k(|A_{\mu}|),
	\end{equation}
	for a positive constant $\sigma$ depending on the initial profile $f_{0}(\alpha)$.
\end{thm}

In the 2D case, we analogously have the following:

\begin{thm}[Existence and Uniqueness in 2D]\label{2Dthm}
	Let $f_{0}\in \dot{\mathcal{F}}^{1,1}\cap L^{2}$ satisfy the bound
	$$\|f_{0}\|_{\mathcal{\dot{F}}^{1,1}} < c(|A_{\mu}|)$$
	for a constant $c(|A_{\mu}|)$ depending on the Atwood number $A_{\mu}$. Then there exists a unique solution to (\ref{graph2d}-\ref{t2d}) with $f\in L^{\infty}(0,T;\mathcal{\dot{F}}^{1,1}\cap L^{2})\cap L^{1}(0,T;\dot{\mathcal{F}}^{2,1})$ such that $f(\alpha,0) = f_{0}(\alpha)$ and
	\begin{equation*}
	\|f\|_{L^2}(t)\leq \|f_0\|_{L^2},\quad \|f\|_{\mathcal{\dot{F}}^{1,1}}(t) + \sigma \int_{0}^{t}\|f\|_{\mathcal{\dot{F}}^{2,1}}(\tau)d\tau\leq \|f_{0}\|_{\mathcal{\dot{F}}^{1,1}} < c(|A_{\mu}|),
	\end{equation*}
	for a positive constant $\sigma$ depending on the initial profile $f_{0}(\alpha)$,
		\begin{multline*}
	\sigma(\|f_0\|_{\foneone})=-1+\nu+\frac{2\|f_0\|_{\foneone}^2 \left(3-\|f_0\|_{\foneone}^2\right)}{\left(1-\|f_0\|_{\foneone}^2\right)^2}\\
	\quad+A_\mu\frac{2\|f_0\|_{\foneone}\left(2A_\mu\|f_0\|_{\foneone}^5-6\|f_0\|_{\foneone}^4-8A_\mu\|f_0\|_{\foneone}^3+4\|f_0\|_{\foneone}^2-2A_\mu \|f_0\|_{\foneone}+2\right)}{\left(1-\|f_0\|_{\foneone}^2\right)^2\left(1-\|f_0\|_{\foneone}^2-2A_\mu \|f_0\|_{\foneone}\right)^2}
	\end{multline*}
	Computing the constant explicitly for $|A_\mu|=1$, we obtain $c(1)\approx 0.128267$.  
\end{thm}

As noted in the introductory section, in the 3D setting, when $A_{\mu} = 0$, the constant $k(0)$ matches the size of the initial data in the 2D without viscosity jump proven in \cite[Remark 5.4]{CCGRS16}, and therefore, improves the size of the initial data in the 3D case without viscosity jump given by \cite{CCGRS16}.

In the graph below the 3D constant $k(|A_\mu|)$ is pictured with respect to $|A_\mu|$. The maximum is $k_0\approx 0.362606$ and the minimum is $k(1)\approx0.080604$.  The graph in the Figure arises from estimating the size of initial data, $k(|A_\mu|)$, needed to satisfy the positivity condition \eqref{condition} of the high order rational polynomial given in the proof.

\begin{figure}[h]\centering
	\includegraphics[width=.7\textwidth]{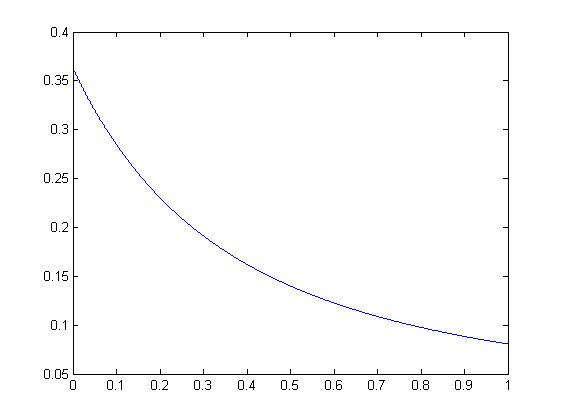}\vspace{-0.4cm}
	\caption{$k(|A_\mu|)$}\label{fig:pic}
\end{figure}

To prove Theorem \ref{Global} and in particular \eqref{globalineq}, we first need to prove apriori estimates on the vorticity and potential jump functions. These estimates on $\|\omega_{i}\|_{\dot{\mathcal{F}}^{s,1}}$ for $s=0,1$ are computed in Section \ref{secw}. The key point of the vorticity estimates is to demonstrate a bound on $\|\omega_{i}\|_{\dot{\mathcal{F}}^{s,1}}$ by a constant multiple of $\|f\|_{\dot{\mathcal{F}}^{s+1,1}}$, as $\omega_{i}(\alpha)$ is of similar regularity to $\nabla f(\alpha)$.

Next, we introduce the following norms with analytic weights:
\bea\label{analyticnorm}
\|f\|_{\mathcal{\dot{F}}^{s,p}_{\nu}}^{p}(t) \eqdef \int |\xi|^{sp}e^{p\nu t |\xi|}|\hat{f}(\xi,t)|^{p} d\xi, \quad \nu>0,\quad s\geq 0,\quad p\geq 1,
\eea
where we also denote $\|f\|_{\mathcal{\dot{F}}^{0,p}_{\nu}}=\|f\|_{\mathcal{F}^{0,p}_{\nu}}$.
In Section \ref{secanalytic}, we will use the vorticity estimates to prove uniform bounds on the analytic weighted quantity
$\|f\|_{\mathcal{\dot{F}}^{1,1}_{\nu}}(t)$:
\begin{thm}[Instant Analyticity]\label{AnalyticTheorem}
	Suppose $f(\alpha,t)$ is a solution to (\ref{decompnonlinear}-\ref{DOmegagraph}) in 3D  with initial data satisfying $\|f_{0}\|_{\mathcal{\dot{F}}^{1,1}} < k(|A_{\mu}|)$ or (\ref{graph2d}-\ref{t2d}) in 2D with initial data satisfying $\|f_{0}\|_{\mathcal{\dot{F}}^{1,1}} < c(|A_{\mu}|)$. Then there exist $\nu=\nu(\|f_{0}\|_{\mathcal{\dot{F}}^{1,1}})>0$ such that the evolution of the quantity $\|f\|_{\mathcal{\dot{F}}^{s,1}_{\nu}}$ satisfies the estimate
	\bea\label{instantanalyticity}
	\|f\|_{\mathcal{\dot{F}}^{1,1}_{\nu}}(t) + \sigma \int_{0}^{t}\|f\|_{\mathcal{\dot{F}}^{2,1}_{\nu}}(\tau)d\tau\leq \|f_{0}\|_{\mathcal{\dot{F}}^{1,1}} 
	\eea
	Furthermore, if $f_0\in L^2$ then
	\bea\label{gainofL2}\|f\|_{\mathcal{F}^{0,2}_{\nu}} \leq \|f_{0}\|_{L^{2}}\exp(R(\|f_{0}\|_{\mathcal{\dot{F}}^{1,1}})),
	\eea
	where $R$ is a positive rational polynomial.
\end{thm}

Setting $\nu = 0$, we obtain the estimate \eqref{globalineq}. Following the instant analyticity argument, we present an $L^{2}$ maximum principle for the Muskat problem with viscosity jump in Section \ref{L2max}. Next, in Section \ref{uniquenesssection}, we give an argument for uniqueness of solutions in the space $\mathcal{F}^{0,1}$, noting that $\mathcal{F}^{0,1} \hookrightarrow L^{\infty}$. All of these apriori estimates finally allow us to perform a regularization argument in Section \ref{regularization}. In this argument, we perform an appropriate mollification of the interface evolution equation for $f(\alpha,t)$ and show that the regularized solutions $f^{\varepsilon_{n}}(\alpha,t)$ converge strongly to $f(\alpha,t)$ in $L^{2}(0,T; \dot{\mathcal{F}}^{1,1})$ and satisfy \eqref{globalineq}. Taking the limit $f^{\varepsilon_{n}}(\alpha,t) \longrightarrow f(\alpha,t)$, we establish the global wellposedness result of Theorem \ref{Global}.

In this paper, we also show analytic results in $L^{2}$ spaces in Section \ref{L2section}. Specifically, we prove uniform bounds on an analytic $L^{2}$ norm, as given by \eqref{gainofL2}
as well as
\bea\label{Hsnudecrease}
\frac{d}{dt}\|f\|_{\dot{\mathcal{F}}^{s,2}_{\nu}}^{2}(t) \leq - \sigma \|f\|_{\dot{\mathcal{F}}^{s+1/2,2}_{\nu}}^{2}
\eea
for $1/2\leq s \leq 3/2$. Note that in general, we denote $\dot{\mathcal{F}}^{s,2}_{\nu}$ by $\dot{H}^{s}_{\nu}$ for $s\neq 0$ and $\mathcal{F}^{0,2}_{\nu}$ by $L^{2}_{\nu}$ throughout the paper. We will use this $L^2$ estimate to show the $L^2$ decay and ill-posedness results later in the paper.

Given solutions with initial data as described in Theorem \ref{Global}, in Section \ref{DecaySection} we obtain large-time decay for solutions to the Muskat problem by using estimates similar to \eqref{instantanalyticity}, \eqref{gainofL2} and \eqref{Hsnudecrease}:

\begin{thm}[Sharp Decay Estimates]\label{DecayTheorem}
	Suppose $f(\alpha,t)$ is a solution to (\ref{decompnonlinear}-\ref{DOmegagraph})  with initial data satisfying $\|f_{0}\|_{\mathcal{\dot{F}}^{1,1}} < k(|A_{\mu}|)$ and $\|f_{0}\|_{L^{2}} < \infty$. Then for any $0 \leq s \leq 1$
	$$\|f\|_{\dot{\mathcal{F}}^{s,1}_{\nu}}(t) \leq C_{s}(1+t)^{-s-1+\lambda},$$
	for arbitrarily small $\lambda > 0$ and some nonnegative constant $C_{s}$ depending on the initial profile $f_{0}(\alpha)$ and the exponent $s$. Moreover, for any $T > 0$, there exists a constant $C_{T,s}$ depending on $f_{0}$, $T$ and $s$ such that
	$$\|f\|_{\dot{H}^{s}_{\nu}}(t) \leq C_{T,s} t^{-s},$$
	for $t > T$.
    In 2D, we have the following decay rate for solutions with initial data satisfying $\|f_0\|_{\foneone}<c(|A_\mu|)$ and $\|f_0\|_{L^2}<\infty$:
    \begin{equation*}
    \|f\|_{\dot{\mathcal{F}}^{s,1}_{\nu}}(t) \leq C_{s}(1+t)^{-s-1/2+\lambda}.
    \end{equation*}
	The $H^s_\nu$ decay rates in 2D are the same.

\end{thm}

\begin{remark}
	We call the decay rates in Theorem \ref{DecayTheorem} sharp for the following reason. If $f_{0}\in \dot{\mathcal{F}}^{1,1} \cap L^{2}$, then it can be seen that $f_{0} \in \dot{\mathcal{F}}^{s',1}$ for $-1<s' \leq 1$ but $f_{0}(\alpha)$ need not be in $\dot{\mathcal{F}}^{-1,1}$. If we consider the linearized contour equation with initial data $\|f_{0}\|_{\dot{\mathcal{F}}^{s',1}}$ for $ -1 <s' < 1$, then for any $s > s'$, we have the equivalence for the linear solutions
	$$\|f_{0}\|_{\dot{\mathcal{F}}^{s',1}} \approx \|t^{s-s'}\|e^{t\Lambda}f_{0}\|_{\dot{\mathcal{F}}^{s,1}}\|_{L^{\infty}_{t}}.$$
	This estimate yields, for example, the optimal rate of $t^{s'-1}$ for decay of $\|f\|_{\dot{\mathcal{F}}^{1,1}}$. Because we at most can guarantee that $\|f_{0}\|_{\dot{\mathcal{F}}^{s',1}} < +\infty$ for $-1 < s' \leq 1$ but not for any lower value of $s$, the decay rates above are sharp. Finally, for $\nu \neq 0$, since $\|f\|_{\dot{\mathcal{F}}^{s,1}} \leq \|f\|_{\dot{\mathcal{F}}^{s,1}_{\nu}}$, the rates are also sharp for the analytic weighted norms.
\end{remark}

In Section \ref{DecaySection}, we additionally note that for $f_{0}$ satisfying the conditions of Theorem \ref{DecayTheorem}, it immediately follows that the solution $f(\alpha,t)$ is in the spaces $\dot{\mathcal{F}}^{s,1}\cap \dot{H}^{s'}$ for any $s > -1$ and $s' \geq 0$. Moreover, due to the decay of the quantity $\|f\|_{\dot{\mathcal{F}}^{1,1}_{\nu}}$, we can show that there exists a constant $k_{s}$ and time $T_{s}$ depending on $s > 1$ and the initial profile $f_{0}$ such that
\bea\label{gainLinfty}
\|f\|_{\mathcal{\dot{F}}^{s,1}_{\nu}}(t) + \sigma_s \int_{T_s}^{t}\|f\|_{\mathcal{\dot{F}}^{s+1,1}_{\nu}}(\tau)d\tau\leq k_{s} 
\eea
for some $\nu > 0$ and for all $t > T_{s}$ for a time $T_{s}$ large enough and depending on $s$ and $f_{0}$. Therefore, we obtain decay rates for $t > T_{s}$:
$$\|f\|_{\mathcal{\dot{F}}^{s,1}_{\nu}}(t) \leq C_{s}t^{-s-1+\lambda}$$
analogously to Theorem \ref{DecayTheorem}. We can draw similar conclusions for the Sobolev norms with analytic weight such as $H^s_\nu$.

Finally, and importantly, we use the $L^{2}_{\nu}$ uniform bound  \eqref{gainofL2} to obtain an ill-posedness argument for the unstable regime of the Muskat problem in Section \ref{illposedsection}:

\begin{thm}[Ill-posedness]\label{IllPTheorem}
	For every $s>0$ and $\epsilon > 0$, there exist a solution $\tilde{f}$ to the unstable regime and $0< \delta < \epsilon$ such that $\|\tilde{f}\|_{H^{s}}(0)<\epsilon$ but $\|\tilde{f}\|_{H^{s}}(\delta) = \infty$.
\end{thm}

This ill-posedness result is very significant because we show instantaneous blow-up of solutions in very low regularity spaces.

We note here that this paper explains the full proof of Theorem \ref{Global} and the other theorems in 3D.  The proof of Theorem \ref{2Dthm} and the other 2D results follow similarly, the 2D results are  actually easier in several places, and for that reason we do not rewrite the 2D proofs in this paper.

\section{A Priori Estimates on $\omega$}\label{secw}

In this section, we will prove the necessary estimates on $\|\omega_{i}\|_{\dot{\mathcal{F}}^{s,1}}$ for $ s = 0,1$ and $i=1,2,3$. These estimates will be used later to prove the bound \eqref{instantanalyticity} on the evolution of a solution in $\|f\|_{\dot{\mathcal{F}}^{1,1}_\nu}$. We first show that $\|\omega_{i}\|_{\mathcal{F}^{0,1}}$ is bounded by quantities depending only on the characteristics of the fluids and $\|f\|_{\dot{\mathcal{F}}^{1,1}}$. Then, using the estimates on $\|\omega_{i}\|_{\mathcal{F}^{0,1}}$, we further show that the quantities $\|\omega_{i}\|_{\dot{\mathcal{F}}^{1,1}}$ for $i=1,2,3$ are linearly bounded by $\|f\|_{\dot{\mathcal{F}}^{2,1}}$ with the linear constant depending on $\|f\|_{\dot{\mathcal{F}}^{1,1}}$.

\begin{prop}\label{f01omega}
	Given the constants $S_1$, $C_1$, $C_2$ depending on $A_\mu$, $\|f\|_{\dot{\mathcal{F}}^{1,1}}$ that are defined by
	\begin{equation}\label{C1C2}
	\begin{aligned}
	S_1=\frac{\yo}{1-\yo^2}, \hspace{0.5cm}C_1=\frac{1-A_\mu S_1}{1-5A_\mu S_1},\hspace{0.5cm}
	C_2=\frac{C_1}{(1-2A_\mu S_1)(1-\yo^2)},
	\end{aligned}
	\end{equation}
	we have the following estimates
	\begin{equation}
	\label{partialOmegaineq1}
	\|\omega_{1}\|_{\mathcal{F}^{0,1}}  = \|\partial_{\alpha_{2}}\Omega\|_{\mathcal{F}^{0,1}}  \leq  2C_1A_\rho  \|f\|_{\dot{\mathcal{F}}^{1,1}} ,
	\end{equation}
	\begin{equation}
	\label{partialOmegaineq2}
	\|\omega_{2}\|_{\mathcal{F}^{0,1}}  = \|\partial_{\alpha_{1}}\Omega\|_{\mathcal{F}^{0,1}}  \leq 2C_1A_\rho  \|f\|_{\dot{\mathcal{F}}^{1,1}},
	\end{equation}
	and
	\begin{equation}
	\begin{aligned}
	\label{om3est}
	\|\omega_{3}\|_{\mathcal{F}^{0,1}} &\leq  12A_\mu A_\rho C_2 \|f\|_{\dot{\mathcal{F}}^{1,1}}^3,\\
	\|\partial_{\alpha_i}\md\|_{\fzerone}&\leq 6 A_\rho C_2 \yo^2,\hspace{0.5cm}i=1,2.
	\end{aligned}
	\end{equation}
	For the potential jump function $\Omega$, we moreover have the estimate
	\begin{equation}
	\label{OmegaF11}
	\|\Omega\|_{\dot{\mathcal{F}}^{1,1}} \leq 2A_{\rho}B_{1}\|f\|_{\foneone},
	\end{equation}
	where
	\begin{equation}\label{B1}
	B_{1} = \frac{1-2A_{\mu}S_{1}}{1-8A_{\mu}S_{1}}.
	\end{equation}
\end{prop}
\begin{proof}
	First, by formulas \eqref{omegai} and \eqref{Omegagraph} we have that 
	\begin{equation}\label{boundomega12}
	\|\omega_1\|_{\mathcal{F}^{0,1}}=\|\partial_{\alpha_2}\Omega\|_{\mathcal{F}^{0,1}}, \hspace{1cm} \|\omega_2\|_{\fzerone}=\|\partial_{\alpha_1}\Omega\|_{\mathcal{F}^{0,1}},
	\end{equation}
	and 
	\begin{equation}\label{boundomega3}
	\begin{aligned}
	\|\omega_3\|_{\mathcal{F}^{0,1}}&=\|\partial_{\alpha_2}\Omega \partial_{\alpha_1}f-\partial_{\alpha_1}\Omega \partial_{\alpha_2}f\|_{\mathcal{F}^{0,1}}=A_\mu\|\partial_{\alpha_2}\mathcal{D} \partial_{\alpha_1}f-\partial_{\alpha_1}\mathcal{D} \partial_{\alpha_2}f\|_{\mathcal{F}^{0,1}}\\
	&\leq A_\mu\|f\|_{\foneone}\left(\|\partial_{\alpha_1}\mathcal{D}\|_{\fzerone}+\|\partial_{\alpha_2}\mathcal{D}\|_{\fzerone}\right),
	\end{aligned}
	\end{equation}
	so it suffices to bound the quantities $\|\partial_{\alpha_{i}}\Omega\|_{\mathcal{F}^{0,1}}$ and $\|\partial_{\alpha_{i}}\md\|_{\mathcal{F}^{0,1}}$ for $i = 1,2$. 
	Notice that from \eqref{Omegagraph} and \eqref{partialD} we have that
	\begin{equation}\label{Df01}
	\begin{aligned}
	\|\partial_{\alpha_i}\Omega\|_{\fzerone}&\leq A_\mu \|\partial_{\alpha_i}\md\|_{\fzerone}+2A_\rho \|f\|_{\foneone},\\
	\|\partial_{\alpha_1}\md\|_{\fzerone}& \leq 2 \|BR_1\|_{\fzerone}+2 \|BR_3\pone f\|_{\fzerone},\\
	\|\partial_{\alpha_2}\md\|_{\fzerone}& \leq 2 \|BR_2\|_{\fzerone}+2 \|BR_3\ptwo f\|_{\fzerone}.
	\end{aligned}
	\end{equation}
	Thus, we proceed to bound the terms $\|BR_1\|_{\fzerone}$, $\|BR_2\|_{\fzerone}$ and $\|BR_3\|_{\fzerone}$, given by \eqref{br1}, \eqref{br2} and \eqref{br3}. 
	We start with the term $\|BR_{1}\|_{\mathcal{F}^{0,1}}$. 
	We first need to bound the Fourier transform of the Birkhoff-Rott integral terms. For the first term in $BR_{1}$,
	\bea\label{br11}
	BR_{11}(f)(\alpha) =  \frac{-1}{4\pi}\text{p.v.}\int_{\mathbb{R}^{2}}\frac{\frac{\beta_{2}}{|\beta|}\omega_{3}(\alpha-\beta)}{(1+\Delta_{\beta}f(\alpha)^{2})^{\frac{3}{2}}} \frac{d\beta}{|\beta|^{2}},
	\eea
	we first apply a change of variables in $\beta$.
	$$ BR_{11}(f)(\alpha) = \frac{1}{4\pi}\text{p.v.}\int_{\mathbb{R}^{2}}\frac{\frac{\beta_{2}}{|\beta|}\omega_{3}(\alpha+\beta)}{(1+\Delta_{-\beta}f(\alpha)^{2})^{\frac{3}{2}}} \frac{d\beta}{|\beta|^{2}}.$$
	Hence, 
	$$BR_{11}(f)(\alpha) = \frac{-1}{8\pi}\Big(\text{p.v.}\int_{\mathbb{R}^{2}}\frac{\frac{\beta_{2}}{|\beta|}\omega_{3}(\alpha-\beta)}{(1+\Delta_{\beta}f(\alpha)^{2})^{\frac{3}{2}}} \frac{d\beta}{|\beta|^{2}} - \text{p.v.}\int_{\mathbb{R}^{2}}\frac{\frac{\beta_{2}}{|\beta|}\omega_{3}(\alpha+\beta)}{(1+\Delta_{-\beta}f(\alpha)^{2})^{\frac{3}{2}}} \frac{d\beta}{|\beta|^{2}}\Big).$$
	By using the Taylor series expansion for the denominator, given by
	\begin{equation*}
	\frac{1}{(1+x^{2})^{\frac{3}{2}}} = \sum_{n=0}^{\infty}a_n (-1)^{n}x^{2n},\quad\mbox{where}\quad  a_n= \frac{(2n+1)!}{(2^{n}n!)^{2}},
	\end{equation*}
	we obtain that
	\begin{multline*}
	BR_{11}(f)(\alpha) = \frac{-1}{8\pi}\sum_{n\geq 0} (-1)^{n} a_{n}\int_{\mathbb{R}^{2}}\!\Big(\frac{\beta_{2}}{|\beta|}\omega_{3}(\alpha-\beta)\Delta_{\beta}f(\alpha)^{2n}\\ - \frac{\beta_{2}}{|\beta|}\omega_{3}(\alpha+\beta)\Delta_{-\beta}f(\alpha)^{2n} \Big)\frac{d\beta}{|\beta|^{2}}.
	\end{multline*}
	Applying the Fourier transform, the products are transformed to convolutions:
	\begin{multline*}
	\widehat{BR_{11}}(\xi) = \frac{-1}{8\pi}\sum_{n\geq 0} (-1)^{n} a_{n}\int_{\mathbb{R}^{2}}\!\frac{\beta_{2}}{|\beta|}\Big(\hat{\omega_{3}}(\xi)e^{-i\beta\cdot\xi}\ast \big(\ast^{2n}\hat{f}(\xi)m(\xi,\beta)\big)\\ -\hat{\omega_{3}}(\xi)e^{i\beta\cdot\xi}\ast \big(\ast^{2n}\hat{f}(\xi)m(\xi,-\beta)\big) \Big)\frac{d\beta}{|\beta|^{2}},
	\end{multline*}
	where $$ m(\xi,\beta) = \frac{1-e^{-i\beta\dot\xi}}{|\beta|}.$$
	Writing the integral in polar coordinates with $\beta = r u $ and $ u = (\cos(\theta),\sin(\theta))$,
	\begin{multline*}
	\widehat{BR_{11}}(\xi) = \frac{-1}{8\pi}\sum_{n\geq 0} (-1)^{n} a_{n}\int_{-\pi}^{\pi}\int_{0}^{\infty}\!\sin(\theta)\Big(\hat{\omega_{3}}(\xi)e^{-ir u\cdot\xi}\ast \big(\ast^{2n}\hat{f}(\xi)m(\xi,r,u)\big)\\ -\hat{\omega_{3}}(\xi)e^{ir u\cdot\xi}\ast \big(\ast^{2n}\hat{f}(\xi)m(\xi,r,-u)\big) \Big)\frac{dr}{r}d\theta.
	\end{multline*}
	By a change of variables in the radial variable,
	
	\begin{multline*}
	\widehat{BR_{11}}(\xi) = \frac{-1}{8\pi}\sum_{n\geq 0} (-1)^{n} a_{n}\int_{-\pi}^{\pi}\int_{0}^{-\infty}\!\sin(\theta)\Big(\hat{\omega_{3}}(\xi)e^{ir u\cdot\xi}\ast \big(\ast^{2n}\hat{f}(\xi)m(\xi,-r,u)\big)\\ -\hat{\omega_{3}}(\xi)e^{-ir u\cdot\xi}\ast \big(\ast^{2n}\hat{f}(\xi)m(\xi,-r,-u)\big) \Big)\frac{-dr}{-r}d\theta.
	\end{multline*}
	Note that $m(\xi,-r,u) =- m(\xi,r,-u)$, 
	and hence, we obtain
	\begin{multline*}
	\widehat{BR_{11}}(\xi) = \frac{-1}{8\pi}\sum_{n\geq 0} (-1)^{n} a_{n}\int_{-\pi}^{\pi}\int_{-\infty}^{0}\!\sin(\theta)\Big(\hat{\omega_{3}}(\xi)e^{-ir u\cdot\xi}\ast \big(\ast^{2n}\hat{f}(\xi)m(\xi,r,u)\big)\\ -\hat{\omega_{3}}(\xi)e^{ir u\cdot\xi}\ast \big(\ast^{2n}\hat{f}(\xi)m(\xi,r,-u)\big) \Big)\frac{dr}{r}d\theta.
	\end{multline*}
	Thus, adding the upper and lower integrals together we obtain 
	\begin{multline*}
	\widehat{BR_{11}}(\xi) = \frac{-1}{16\pi}\sum_{n\geq 0} (-1)^{n} a_{n}\int_{-\pi}^{\pi}\int_{-\infty}^{\infty}\!\sin(\theta)\Big(\hat{\omega_{3}}(\xi)e^{-ir u\cdot\xi}\ast \big(\ast^{2n}\hat{f}(\xi)m(\xi,r,u)\big)\\ -\hat{\omega_{3}}(\xi)e^{ir u\cdot\xi}\ast \big(\ast^{2n}\hat{f}(\xi)m(\xi,r,-u)\big) \Big)\frac{dr}{r}d\theta.
	\end{multline*}
	Writing out of the convolutions in integral form and using the equality
	$$m(\xi, r, u) = iu\cdot\xi\int_{0}^{1}e^{-ir(1-s)u\cdot\xi} ds,$$ we obtain that
	\begin{multline*}
	\widehat{BR_{11}}(\xi)\\ = \frac{-1}{16\pi}\sum_{n\geq 0}  a_{n}\int_{\mathbb{R}^{2}}\cdots\int_{\mathbb{R}^{2}}\int_{-\pi}^{\pi}\sin(\theta)\hat{\omega_{3}}(\xi-\xi_{1})\prod_{j=1}^{2n-1}(u\cdot (\xi_{j}-\xi_{j+1}))(u\cdot \xi_{2n})\hat{f}(\xi_{j}-\xi_{j+1})\\\int_{0}^{1}\cdots\int_{0}^{1} \int_{-\infty}^{\infty}\Big(e^{-iAr} - e^{iAr} \Big)\frac{dr}{r} ds_{1}\cdots ds_{2n} d\theta d\xi_{1}\cdots d\xi_{2n+1}.
	\end{multline*}
	where $$A = u\cdot (\xi-\xi_{1}) + \sum_{j=1}^{2n-1}(1-s_{j})u\cdot(\xi_{j}-\xi_{j+1}) + u\cdot \xi_{2n}.$$
	Next, notice that
	\begin{multline*}
	\Big|\int_{0}^{1}\cdots\int_{0}^{1} \int_{-\infty}^{\infty}\Big(e^{-iAr} - e^{iAr} \Big)\frac{dr}{r} ds_{1}\cdots ds_{2n}\Big|\\
	\leq \pi\Big|\int_{0}^{1}\cdots\int_{0}^{1} sgn(A)-sgn(-A) ds_{1}\cdots ds_{2n}\Big|
	\leq 2\pi.
	\end{multline*}
	Moreover, if $\xi = (\xi^{(1)},\xi^{(2)})$, then
	$$|u\cdot \xi| = |\cos(\theta)\xi^{(1)}+\sin(\theta)\xi^{(2)}| = |\xi||\sin(\theta+\alpha)|,$$
	where $\alpha$ satisfies $\sin(\alpha) = \xi^{(1)}/|\xi|$, and therefore, $\cos(\alpha) = \xi^{(2)}/|\xi|$. Using these estimates,
	$$
	|\widehat{BR_{11}}|(\xi) \leq \frac{1}{8}\sum_{n\geq 0}  a_{n}\Big( (\ast^{2n}|\cdot||\hat{f}(\cdot)|)\ast |\hat{\omega_{3}}(\cdot)|\Big)(\xi)\int_{-\pi}^{\pi} |\sin(\theta)|\prod_{j=1}^{2n}|\sin(\theta + \alpha_{j})| d\theta$$
	for some angles $\alpha_{j}$. Finally, note that
	$$\int_{-\pi}^{\pi} |\sin(\theta)|\prod_{j=1}^{2n}|\sin(\theta + \alpha_{j})| d\theta \leq \int_{-\pi}^{\pi} |\sin(\theta)|^{2n+1} d\theta = 4/a_{n}.$$
	Summarizing,
	\bea\label{br11xi}
	|\widehat{BR_{11}}|(\xi) \leq \frac{1}{2}\sum_{n\geq 0} \Big( (\ast^{2n}|\cdot||\hat{f}(\cdot)|)\ast |\widehat{\omega}_{3}(\cdot)|\Big)(\xi).
	\eea
	The estimates on $BR_{12}$, $BR_{2}$ and $BR_3$ follow as the one on $BR_{11}$.
	We conclude that
	\begin{equation*}
	\begin{aligned}
	\|BR_1\|_{\fzerone}&\leq \frac12\frac{1}{1-\yo^2}\left(\yo \|\omega_2\|_{\fzerone}+\|\omega_3\|_{\fzerone}\right),\\
	\|BR_2\|_{\fzerone}&\leq \frac12\frac{1}{1-\yo^2}\left(\yo \|\omega_1\|_{\fzerone}+\|\omega_3\|_{\fzerone}\right),\\
	\|BR_3\|_{\fzerone}&\leq \frac12\frac{1}{1-\yo^2}\left(\|\omega_1\|_{\fzerone}+\|\omega_2\|_{\fzerone}\right),\\
	\end{aligned}
	\end{equation*}
	Introducing this bounds into \eqref{Df01} we find that
	\begin{equation*}
	\|\pone \md\|_{\fzerone}\leq \frac{\yo}{1-\yo^2}\left(2\|\omega_2\|_{\fzerone}+\|\omega_1\|_{\fzerone}\right)+\frac{1}{1-\yo^2}\|\omega_3\|_{\fzerone}.
	\end{equation*}
	Substituting the bounds for the vorticity \eqref{boundomega12},\eqref{boundomega3}, it follows that
	\begin{equation*}
	\begin{aligned}
	\|\pone \md\|_{\fzerone}&\leq \frac{\yo}{1-\yo^2}\left(2\|\pone \Omega\|_{\fzerone}+\|\ptwo \Omega\|_{\fzerone}\right)\\
	&\quad+A_\mu\frac{\yo}{1-\yo^2}\left(\|\pone \md\|_{\fzerone}+\|\ptwo \md\|_{\fzerone}\right).
	\end{aligned}
	\end{equation*}
	Analogously, 
	\begin{equation*}
	\begin{aligned}
	\|\ptwo \md\|_{\fzerone}&\leq \frac{\yo}{1-\yo^2}\left(2\|\ptwo \Omega\|_{\fzerone}+\|\pone \Omega\|_{\fzerone}\right)\\
	&\quad+A_\mu\frac{\yo}{1-\yo^2}\left(\|\pone \md\|_{\fzerone}+\|\ptwo \md\|_{\fzerone}\right).
	\end{aligned}
	\end{equation*}
	If we denote 
	\begin{equation}
	\label{S1S2}
	S_1=\frac{\yo}{1-\yo^2},\hspace{1cm}S_2=\frac{S_1}{1-A_\mu S_1},
	\end{equation}
	the above inequalities can be written as
	\begin{equation}
	\begin{aligned}
	\|\pone\md\|_{\fzerone}&\leq S_2 \left( 2\|\pone \Omega\|_{\fzerone} +\|\ptwo \Omega\|_{\fzerone} +A_\mu \|\ptwo \md\|_{\fzerone} \right),\\
	\|\ptwo\md\|_{\fzerone}&\leq S_2 \left( 2\|\ptwo \Omega\|_{\fzerone} +\|\pone \Omega\|_{\fzerone} +A_\mu \|\pone \md\|_{\fzerone} \right).
	\end{aligned}
	\end{equation}
	Therefore, it is not hard to see that
	\begin{equation}
	\label{Df01bound}
	\begin{aligned}
	\|\pone \md\|_{\fzerone}\leq \frac{S_2(2+A_\mu S_2 ) }{1-(A_\mu S_2)^2}\Big( \|\pone \Omega\|_{\fzerone}+\|\ptwo \Omega\|_{\fzerone} \Big),\\
	\|\ptwo \md\|_{\fzerone}\leq \frac{S_2(2+A_\mu S_2 )}{1-(A_\mu S_2)^2}\Big( \|\ptwo \Omega\|_{\fzerone}+\|\pone \Omega\|_{\fzerone} \Big).
	\end{aligned}
	\end{equation}
	By defining the following constants
	\begin{equation*}
	c_1=\frac{S_2}{1-(A_\mu S_2)^2}(2+A_\mu S_2),\hspace{0.7cm}c_2=\frac{S_2}{1-(A_\mu S_2)^2}(1+2A_\mu S_2),
	\end{equation*}
	and recalling \eqref{Df01} and the bounds above we have that
	\begin{equation*}
	\begin{aligned}
	\|\pone \Omega\|_{\fzerone}\leq A_\mu c_1\|\pone \Omega\|_{\fzerone}+A_\mu c_2\|\ptwo \Omega\|_{\fzerone}+2A_\rho\yo,
	\end{aligned}
	\end{equation*}
	\begin{equation*}
	\begin{aligned}
	\|\ptwo \Omega\|_{\fzerone}\leq A_\mu c_1\|\ptwo \Omega\|_{\fzerone}+A_\mu c_2\|\pone \Omega\|_{\fzerone}+2A_\rho\yo.
	\end{aligned}
	\end{equation*}
	Therefore, we can conclude that
	\begin{equation*}
	\|\partial_{\alpha_i}\Omega\|_{\fzerone}\leq 2A_\rho \yo\frac{1}{1-A_\mu(c_1+c_2)}.
	\end{equation*}
	This expression can be simplified further to obtain that
	\begin{equation*}
	\|\partial_{\alpha_i}\Omega\|_{\fzerone}\leq 2A_\rho \yo\frac{1-A_\mu S_1}{1-5A_\mu S_1}.
	\end{equation*}
	Going back to \eqref{Df01bound} we find that
	\begin{equation*}
	\begin{aligned}
	\|\partial_{\alpha_i}\md\|_{\fzerone}&\leq \frac{S_2}{1-(A_\mu S_2)^2}3(1+A_\mu S_2 )2A_\rho \yo\frac{1-A_\mu S_1}{1-5A_\mu S_1}\\
	&=6A_\rho \yo \frac{S_1(1-A_\mu S_1)}{(1-2A_\mu S_1)(1-5A_\mu S_1)}
	\end{aligned}
	\end{equation*}
	This last two bounds combined with the estimates \eqref{boundomega12}, \eqref{boundomega3} conclude the proof of \eqref{partialOmegaineq1}-\eqref{om3est}. Finally, to show \eqref{OmegaF11}, we do the following using \eqref{Df01bound}:
	\begin{align*}
	\|\Omega\|_{\dot{\mathcal{F}}^{1,1}} &\leq A_{\mu}\|\partial_{\alpha_{1}}\mathcal{D}(\Omega)\|_{\mathcal{F}^{0,1}} + A_{\mu}\|\partial_{\alpha_{2}}\mathcal{D}(\Omega)\|_{\mathcal{F}^{0,1}} + 2A_{\rho}\|f\|_{\dot{\mathcal{F}}^{1,1}}\\
	&\leq \frac{6A_{\mu}S_2}{1-A_\mu S_2}\|\Omega\|_{\dot{\mathcal{F}}^{1,1}} + 2A_{\rho}\|f\|_{\dot{\mathcal{F}}^{1,1}}.
	\end{align*}
	Therefore,
	$$\|\Omega\|_{\dot{\mathcal{F}}^{1,1}} \leq 2A_{\rho}\Big(\frac{1-2A_{\mu}S_{1}}{1-8A_{\mu}S_{1}}\Big)\|f\|_{\dot{\mathcal{F}^{1,1}}}.$$
	This concludes the proof.
\end{proof}

\begin{prop}\label{f11vorticity}
	Define the constants $C_3$, $C_4$ and $C_5$ depending on $A_\mu$ and $\|f\|_{\dot{\mathcal{F}}^{1,1}}$,
	\begin{equation}\label{C5}
	\begin{aligned}
	C_3&=\frac{1+\yo^2}{1-\yo^2}\left(1+A_\mu\frac{6\yo(1-A_\mu S_1)}{(1-\yo^2)(1-2A_\mu S_1)(1-5A_\mu S_1)}\right),\\
	C_4&=\frac{1+S_2^2A_\mu^2\left(C_3+C_1+4S_1C_1\yo\right)}{1-3A_\mu S_2(1+A_\mu S_2)},\\
	C_5&=\frac{S_2}{\yo}\frac{3+A_\mu S_2(3+C_3+C_1+4S_1C_1\yo)}{1-3A_\mu S_2(1+A_\mu S_2)},
	\end{aligned}
	\end{equation}
	where $C_1$, $S_1$ and $S_2$ are given by \eqref{C1C2} and \eqref{S1S2}.
	Then, we have the following estimates
	\begin{equation}
	\label{f11omega}
	\begin{aligned}
	\|\omega_1\|_{\foneone}&=\|\ptwo\Omega\|_{\foneone}\leq 2A_\rho C_4 \|f\|_{\ftwoone},\\
	\|\omega_2\|_{\foneone}&=\|\pone\Omega\|_{\foneone}\leq 2A_\rho C_4 \|f\|_{\ftwoone},
	\end{aligned}
	\end{equation}
	and 
	\begin{equation}
	\label{f11omega3}
	\begin{aligned}
	\|\omega_3\|_{\foneone}&\leq 4A_\mu A_\rho \yo^2\|f\|_{\ftwoone}(C_5+3C_2),\\
	\|\partial_{\alpha_i}\md\|_{\foneone}&\leq 2A_\rho C_5\|f\|_{\foneone}\|f\|_{\ftwoone},\hspace{0.5cm}i=1,2.
	\end{aligned}	
	\end{equation}
	Moreover,
	$$\|\Omega\|_{\ftwoone} \leq 2A_{\rho} B_{2}\|f\|_{\mathcal{\dot{F}}^{2,1}},$$
	where
	\begin{equation}\label{B2}
	B_{2} = \frac{1+2S_{2}^{2}A_{\mu}(C_{1}+C_{3}+ 4S_{1}C_{1}\|f\|_{\dot{\mathcal{F}}^{1,1}})}{1-6A_{\mu}S_{2}(1+A_{\mu}S_{2})}.
	\end{equation}
\end{prop}
\begin{proof}
	Using the formulas for the vorticity it follows that
	\begin{equation*}
	\|\omega_1\|_{\foneone}=\|\ptwo\Omega\|_{\foneone},\hspace{0.5cm}\|\omega_2\|_{\foneone}=\|\pone \Omega\|_{\foneone},
	\end{equation*}
	\begin{equation*}
	\begin{aligned}
	\|\omega_3\|_{\foneone}&\leq A_\mu \yo (\|\ptwo\md\|_{\foneone}\|\pone\md\|_{\foneone})\\
	&\quad+A_\mu \|f\|_{\ftwoone} (\|\pone\md\|_{\fzerone}+\|\ptwo \md\|_{\fzerone}).
	\end{aligned}
	\end{equation*}
	It suffices then to bound $\|\partial_{\alpha_i}\Omega\|_{\foneone}$ and $\|\partial_{\alpha_i}\md\|_{\foneone}$.  From \eqref{partialalphaOmega} we have that
	\begin{equation}
	\label{Omegaf21}
	\|\partial_{\alpha_{1}}\Omega\|_{\dot{\mathcal{F}}^{1,1}}  \leq 2A_{\rho}\|f\|_{\dot{\mathcal{F}}^{2,1}}+ A_\mu \|\pone\md\|_{\foneone},
	\end{equation}
	\begin{equation*}
	\|\pone\md\|_{\foneone}\leq 2 \|BR_{1}\|_{\foneone} +  2\|BR_{3}\|_{\fzerone} \|f\|_{\ftwoone}+2\|BR_{3}\|_{\foneone} \|f\|_{\foneone}.
	\end{equation*}
	Using an analogous bound to \eqref{br11xi}, it follows that
	\begin{multline*}
	\|BR_{1}\|_{\dot{\mathcal{F}}^{1,1}}  = \int_{\mathbb{R}^{2}} |\xi||\widehat{BR_{1}}(\xi)| d\xi  \leq  \frac{1}{2} \sum_{n\geq 0} \int |\xi| |\widehat{\omega_{3}}(\cdot)|\ast(\ast^{2n}|\cdot||\hat{f}(\cdot)|)(\xi) d\xi\\ +
	\frac12 \sum_{n\geq 0} \int |\xi||\widehat{\omega_{2}}(\cdot)|\ast(\ast^{2n+1}|\cdot||\hat{f}(\cdot)|)(\xi) d\xi.
	\end{multline*}
	By the product rule, we can distribute the multiplier $|\xi|$ to each term in the convolution to obtain
	\begin{multline*}
	\|BR_{1}\|_{\dot{\mathcal{F}}^{1,1}}  \leq  \frac12\sum_{n\geq 1} 2n \int |\widehat{\omega_{3}}(\cdot)|\ast(\ast^{2n-1}|\cdot||\hat{f}(\cdot)|\ast |\cdot|^{2}|\hat{f}(\cdot)|)(\xi) d\xi\\+
	\frac12 \sum_{n\geq 0} (2n+1) \int |\widehat{\omega_{2}}(\cdot)|\ast(\ast^{2n}|\cdot||\hat{f}(\cdot)|\ast |\cdot|^{2}|\hat{f}(\cdot)|)(\xi) d\xi\\ +\frac{1}{2}\sum_{n\geq 0}  \int (|\cdot| |\widehat{\omega_{3}}(\cdot)|)\ast(\ast^{2n}|\cdot||\hat{f}(\cdot)|)(\xi) d\xi  \\+
	\frac12 \sum_{n\geq 0}  \int (|\cdot||\widehat{\omega_{2}}(\cdot)|)\ast(\ast^{2n+1}|\cdot||\hat{f}(\cdot)|)(\xi) d\xi.
	\end{multline*}
	Using Young's inequality, we finally obtain that
	\begin{equation*}
	\begin{aligned}
	\|BR_{1}\|_{\dot{\mathcal{F}}^{1,1}} & \leq  \frac{\yo}{\left(1-\yo^2\right)^2}\|\omega_3\|_{\fzerone}\|f\|_{\ftwoone}+\frac{1+\yo^2}{2\left(1-\yo^2\right)^2}\|\omega_2\|_{\fzerone}\|f\|_{\ftwoone}\\
	&\quad+\frac12\frac{\yo}{1-\yo^2}\|\omega_2\|_{\foneone}+\frac12\frac{1}{1-\yo^2}\|\omega_3\|_{\foneone}.
	\end{aligned}
	\end{equation*}
	Proceeding in a similar way we have that
	\begin{multline*}
	\|BR_3\|_{\foneone}\leq \frac12\frac{1}{1-\yo^2}(\|\omega_1\|_{\foneone}+\|\omega_2\|_{\foneone})\\+\frac{\yo}{\left(1-\yo^2\right)^2}\|f\|_{\ftwoone}(\|\omega_1\|_{\fzerone}+\|\omega_2\|_{\fzerone}).
	\end{multline*}
	From the bounds in Proposition \ref{f01omega} we can write the above estimates as follows
	\begin{equation}
	\label{br1f11}
	\begin{aligned}
	\|BR_{1}\|_{\foneone} & \leq  A_\rho \|f\|_{\ftwoone}\yo \frac{\left(1+\yo^2\right)\left(1+6A_\mu C_2\yo\right)}{\left(1-\yo^2\right)^2}\\
	&\quad+\frac12S_1\|\pone \Omega\|_{\foneone}+A_\mu\frac12S_1\left(\|\pone\md\|_{\foneone}+\|\ptwo\md\|_{\foneone}\right).
	\end{aligned}
	\end{equation}
	\begin{equation}\label{br3f11}
	\begin{aligned}
	\|BR_3\|_{\foneone}\leq 4C_1A_\rho S_1^2 \|f\|_{\ftwoone}+ \frac12\frac{1}{1-\yo^2}(\|\ptwo\Omega\|_{\foneone}+\|\pone\Omega\|_{\foneone}).
	\end{aligned}
	\end{equation}

	Then, using \eqref{br1f11} and \eqref{br3f11} as well as the estimates from Proposition \ref{f01omega}, we obtain
	\begin{equation*}
	\begin{aligned}
	\|\pone \md\|_{\foneone}&\leq 2\|BR_1\|_{\foneone}+2\|BR_3\|_{\fzerone}\|f\|_{\ftwoone}+2\|BR_1\|_{\foneone}\|f\|_{\foneone}\\
	&\leq S_1\|\pone \Omega\|_{\foneone}+ A_\mu S_1 \|\pone \md\|_{\foneone}+A_\mu S_1 \|\ptwo\md\|_{\foneone}\\
	&\quad+2C_3 A_\rho \yo \|f\|_{\ftwoone}+2S_1C_1A_\rho \|f\|_{\ftwoone}+8S_1^2C_1A_\rho \yo\|f\|_{\ftwoone}\\
	&\quad+S_1\|\ptwo\Omega\|_{\foneone}+S_1\|\pone\Omega\|_{\foneone}\\
	&\leq 2S_1\|\pone\Omega\|_{\foneone}+S_1\|\ptwo\Omega\|_{\foneone}+A_\mu S_1 \|\pone \md\|_{\foneone}+A_\mu S_1 \|\ptwo\md\|_{\foneone}\\
	&\quad+ 2S_1 A_\rho\|f\|_{\ftwoone}\left(C_3+C_1+4S_1C_1\yo\right)
	\end{aligned}
	\end{equation*}
	Recalling the definition of $S_1$ and $S_2$ \eqref{S1S2}, from here we can write that
	\begin{equation*}
	\begin{aligned}
	\|\pone \md\|_{\foneone}&\leq 2S_2\|\pone\Omega\|_{\foneone}+S_2\|\ptwo\Omega\|_{\foneone}+A_\mu S_2\|\ptwo\md\|_{\foneone}\\
	&\quad+2A_\rho S_2\|f\|_{\ftwoone}\left(C_3+C_1+4S_1C_1\yo\right),
	\end{aligned}
	\end{equation*}
	and analogously,
	\begin{equation*}
	\begin{aligned}
	\|\ptwo \md\|_{\foneone}&\leq 2S_2\|\ptwo\Omega\|_{\foneone}+S_2\|\pone\Omega\|_{\foneone}+A_\mu S_2\|\pone\md\|_{\foneone}\\
	&\quad+2A_\rho S_2\|f\|_{\ftwoone}\left(C_3+C_1+4S_1C_1\yo\right).
	\end{aligned}
	\end{equation*}
	We conclude that
	\begin{equation*}
	\begin{aligned}
	\|\pone \md\|_{\foneone}&\leq S_2(2+A_\mu S_2)\|\pone \Omega\|_{\foneone}+S_2(1+2A_\mu S_2)\|\ptwo \Omega\|_{\foneone}\\
	&\quad+2S_2^2A_\mu A_\rho \|f\|_{\ftwoone}\left(C_3+C_1+4S_1C_1\yo\right),\\
	\|\ptwo \md\|_{\foneone}&\leq S_2(2+A_\mu S_2)\|\ptwo \Omega\|_{\foneone}+S_2(1+2A_\mu S_2)\|\pone \Omega\|_{\foneone}\\
	&\quad+2S_2^2A_\mu A_\rho \|f\|_{\ftwoone}\left(C_3+C_1+4S_1C_1\yo\right).
	\end{aligned}
	\end{equation*}
	Now, we will introduce these inequalities into \eqref{Omegaf21} to close the estimates. First, we have that
	\begin{equation*}
	\begin{aligned}
	\|\pone\Omega\|_{\foneone}&\leq 2A_\rho\left(1+S_2^2A_\mu^2 \left(C_3+C_1+4S_1C_1\yo\right)\right) \|f\|_{\ftwoone}\\
	&\quad+A_\mu S_2(2+A_\mu S_2)\|\pone \Omega\|_{\foneone}+A_\mu S_2(1+2A_\mu S_2)\|\ptwo \Omega\|_{\foneone},
	\end{aligned}
	\end{equation*}
	which implies that
	\begin{equation*}
	\begin{aligned}
	\|\pone\Omega\|_{\foneone}&\leq \frac{A_\mu S_2(1+2A_\mu S_2)}{1-A_\mu S_2(2+A_\mu S_2)}\|\ptwo\Omega\|_{\foneone}\\
	&\quad+2A_\rho\|f\|_{\ftwoone}\frac{1+S_2^2A_\mu^2\left(C_3+C_1+4S_1C_1\yo\right)}{1-A_\mu S_2(2+A_\mu S_2)}.
	\end{aligned}
	\end{equation*}
	The above inequality combined with the analogous one for $\ptwo\Omega$ yields that
	\begin{equation*}
	\begin{aligned}
	\|\partial_{\alpha_i}\Omega\|_{\foneone}&\leq 2A_\rho\|f\|_{\ftwoone} \frac{1}{1-\frac{A_\mu S_2(1+2A_\mu S_2)}{1-A_\mu S_2(2+A_\mu S_2)}}\frac{1+S_2^2A_\mu^2\left(C_3+C_1+4S_1C_1\yo\right)}{1-A_\mu S_2(2+A_\mu S_2)}\\
	&=2A_\rho\|f\|_{\ftwoone} \frac{1+S_2^2A_\mu^2\left(C_3+C_1+4S_1C_1\yo\right)}{1-3A_\mu S_2(1+A_\mu S_2)}.
	\end{aligned}
	\end{equation*}
	By denoting
	\begin{equation*}
	C_4=\frac{1+S_2^2A_\mu^2\left(C_3+C_1+4S_1C_1\yo\right)}{1-3A_\mu S_2(1+A_\mu S_2)},
	\end{equation*}
	we conclude that
	\begin{equation}
	\|\partial_{\alpha_i}\Omega\|_{\foneone}\leq 2A_\rho C_4\|f\|_{\ftwoone},
	\end{equation}
	and therefore
	\begin{equation*}
	\begin{aligned}
	\|\partial_{\alpha_i}\md\|_{\foneone}&\leq 2A_\rho S_2 \|f\|_{\ftwoone}\Big(3(1+A_\mu S_2)C_4+A_\mu S_2 (C_3+C_1+4S_1C_1\yo)\Big)\\
	&=2A_\rho C_5\|f\|_{\foneone}\|f\|_{\ftwoone},
	\end{aligned}
	\end{equation*}
	where we have denoted
	\begin{equation*}
	C_5=\frac{S_2}{\yo}\frac{3+A_\mu S_2(3+C_3+C_1+4S_1C_1\yo)}{1-3A_\mu S_2(1+A_\mu S_2)}.
	\end{equation*}
	Thus, the estimates for the vorticity are
	\begin{equation*}
	\begin{aligned}
	\|\omega_i\|_{\foneone}&\leq 2A_\rho C_4 \|f\|_{\ftwoone},\hspace{0.5cm}i=1,2,\\
	\|\omega_3\|_{\foneone}&\leq 4A_\mu A_\rho \yo^2\|f\|_{\ftwoone}(C_5+3C_2).
	\end{aligned}
	\end{equation*}
	Finally, we estimate the quantity $\|\Omega\|_{\dot{\mathcal{F}}^{2,1}}$.
	\begin{align*}
	\|\Omega\|_{\dot{\mathcal{F}}^{2,1}} &\leq A_{\mu}\|\partial_{\alpha_{1}}\mathcal{D}(\Omega)\|_{\dot{\mathcal{F}}^{1,1}} + A_{\mu}\|\partial_{\alpha_{2}}\mathcal{D}(\Omega)\|_{\dot{\mathcal{F}}^{1,1}} + 2A_{\rho}\|f\|_{\dot{\mathcal{F}}^{2,1}}\\
	&\leq 6A_{\mu}S_{2}(1+A_{\mu}S_{2}) + 2A_{\rho}(1+2S_{2}^{2}A_{\mu}(C_{1}+C_{3}+ 4S_{1}C_{1}\|f\|_{\dot{\mathcal{F}}^{1,1}}))\|f\|_{\dot{\mathcal{F}}^{2,1}}.
	\end{align*}
	Therefore,
	$$\|\Omega\|_{\ftwoone} \leq 2A_{\rho}\frac{1+2S_{2}^{2}A_{\mu}(C_{1}+C_{3}+ 4S_{1}C_{1}\|f\|_{\dot{\mathcal{F}}^{1,1}})}{1-6A_{\mu}S_{2}(1+A_{\mu}S_{2})}\|f\|_{\dot{\mathcal{F}}^{2,1}}.$$
	This concludes the proof.
\end{proof}

\begin{remark}\label{vorticityremark}
	Because we actually have the triangle inequality
	$$|\xi|^{s} \leq |\xi-\xi_{1}|^{s}+ \sum_{k=1}^{m}|\xi_{j}-\xi_{j+1}|^{s} + |\xi_{m+1}|^{s}$$
	for all $0<s\leq 1$, notice that the same arguments as above can be used to show that
	$$\|\omega_{1}\|_{\dot{\mathcal{F}}^{s,1}_{\nu}} =\|\partial_{\alpha_{2}}\Omega\|_{\dot{\mathcal{F}}^{s,1}_{\nu}} \leq 2 A_\rho C_{4,\nu}\|f\|_{\dot{\mathcal{F}}^{s+1,1}_{\nu}}$$
	and
	$$\|\omega_{3}\|_{\dot{\mathcal{F}}^{s,1}_{\nu}}  \leq 4A_\mu A_\rho \yo^2\|f\|_{\ftwoone}(C_{5,\nu}+3C_{2,\nu})$$
	where the constants $C_{2,\nu}$, $C_{4,\nu}$ and $C_{5,\nu}$ now depend on $\|f\|_{\dot{\mathcal{F}}^{1,1}_{\nu}}$ rather than $\|f\|_{\dot{\mathcal{F}}^{1,1}}$. 
\end{remark}

\section{Instant Analyticity of $f$}\label{secanalytic}

We dedicate this section to proving the norm decrease inequality \eqref{normdecrease} which will be needed to obtain the global existence results of this paper. Note that \eqref{normdecrease} states that the interface function becomes instantly analytic given medium-sized initial data $f_{0} \in \dot{\mathcal{F}}^{1,1}$. Precisely, we show the following:

\begin{prop}\label{normdecreaseprop}
	Assume the initial data $f_0$ satisfies that 
	\begin{equation}\label{condition}
	\sigma\left(\|f_0\|_{\mathcal{\dot{F}}^{1,1}}\right)>0,
	\end{equation}
	where
	\begin{equation*}
	\begin{aligned}
	\sigma\left(\|f_0\|_{\mathcal{\dot{F}}^{1,1}}\right)& = -\nu + A_{\rho}\left(1 - 2 \Big(\frac{2B_{1}+B_{2}-B_{2}\|f_0\|_{\dot{\mathcal{F}}^{1,1}}^{2}}{(1-\|f_0\|_{\dot{\mathcal{F}}^{1,1}}^{2})^{2}}\Big)\|f_0\|_{\dot{\mathcal{F}}^{1,1}}^{2}\right.\\
	&\left.\quad- A_{\mu}\Big(\frac{12C_{2}+2C_{5}-2C_{5}\|f_0\|_{\dot{\mathcal{F}}^{1,1}}^{2}}{(1-\|f_0\|_{\dot{\mathcal{F}}^{1,1}}^{2})^{2}}\Big)\|f_0\|_{\dot{\mathcal{F}}^{1,1}}^{3} - 2A_{\mu}C_{5}\|f_0\|_{\dot{\mathcal{F}}^{1,1}}\right).
	\end{aligned}
	\end{equation*}
	All the constants above are defined precisely in  \eqref{C1C2}, \eqref{B1}, \eqref{C5}, and \eqref{B2}, which are given during the proofs of the previous estimates.  Then 
	\begin{equation*}
	\|f\|_{\dot{\mathcal{F}}^{1,1}_{\nu}}(t)+ \sigma\left(\|f_0\|_{\mathcal{\dot{F}}^{1,1}}\right) \int_0^t \|f\|_{\dot{\mathcal{F}}^{2,1}_{\nu}}(\tau) d\tau \leq \|f_0\|_{\mathcal{\dot{F}}^{1,1}}.
	\end{equation*}
\end{prop}

\begin{proof}
	We will use the evolution equation \eqref{decompnonlinear} and \eqref{N1N2N3}. Differentiating the quantity $\|f\|_{\dot{\mathcal{F}}^{1,1}_{\nu}}$, we obtain
	\begin{align*}
	\frac{d}{dt}\|f\|_{\dot{\mathcal{F}}^{1,1}_{\nu}} (t) &= \frac{d}{dt} \Bigg(\int |\xi| e^{\nu t |\xi|} |\hat{f}(\xi)| d\xi \Bigg)\\
	&\leq \nu \int |\xi|^{2} e^{t \nu |\xi|} |\hat{f}(\xi)| d\xi + \int |\xi| e^{\nu t|\xi|} \frac{1}{2} \Big(\frac{\hat{f}_{t}\overline{\hat{f}}+\hat{f}\overline{\hat{f}_{t}}}{|\hat{f}(\xi)|}\Big) d\xi\\
	&\leq  (\nu-A_{\rho}) \int |\xi|^{2} e^{ \nu t|\xi|} |\hat{f}(\xi)| d\xi + \int |\xi| e^{\nu t|\xi|} |\widehat{N(f)}(\xi)| d\xi.
	\end{align*}
	Hence, using the decomposition \eqref{N1N2N3}, we can use the Fourier arguments as earlier, such as \eqref{br11xi}, to pointwise bound the nonlinear term $$|\widehat{N(f)}(\xi)| \leq |\widehat{N_{1}(f)}(\xi)| + |\widehat{N_{2}(f)}(\xi)| + |\widehat{N_{3}(f)}(\xi)|$$ in frequency space. The latter two terms are bounded by
	\bea\label{N2est}
	|\widehat{N_{2}(f)}(\xi)| \leq \sum_{n\geq 1}\Big((\ast^{2n}|\cdot||\hat{f}(\cdot)|)\ast |\cdot||\hat{\Omega}(\xi)|\Big)(\xi)
	\eea
	and
	\bea\label{N3est}
	|\widehat{N_{3}(f)}(\xi)| \leq \frac{A_{\mu}}{2}\sum_{n\geq 1}\Big((\ast^{2n}|\cdot||\hat{f}(\cdot)|)\ast |\cdot||\widehat{\mathcal{D}(\Omega)}(\xi)|\Big)(\xi).
	\eea
	The estimate on $\widehat{N_{1}(f)}(\xi)$ is done in Section \ref{secw}:
	$$\|N_{1}\|_{\dot{\mathcal{F}}^{1,1}_{\nu}} = \frac{A_{\mu}}{2}\|\mathcal{D}(\Omega)\|_{\dot{\mathcal{F}}^{2,1}_{\nu}} \leq 2A_{\rho}A_{\mu}C_{5}\|f\|_{\dot{\mathcal{F}}^{1,1}_{\nu}}\|f\|_{\dot{\mathcal{F}}^{2,1}_{\nu}}.$$
	For the other two nonlinear terms, using the triangle inequality
	$$|\xi|\leq |\xi-\xi_{1}| + |\xi_{1}-\xi_{2}| + \cdots + |\xi_{2n}|,$$ we obtain that
	$$e^{\nu t|\xi|} \leq e^{\nu t|\xi-\xi_{1}|} e^{\nu t|\xi_{1}-\xi_{2}|} \cdots e^{\nu t|\xi_{2n}|}  $$
	and therefore

	\begin{multline*}
	\int |\xi| e^{\nu t|\xi|} |\widehat{N_{2}(f)}(\xi)| d\xi \leq \sum_{n\geq 1} \int |\xi| e^{\nu t|\xi|}\Big((\ast^{2n}|\cdot||\hat{f}(\cdot)|)\ast |\cdot||\hat{\Omega}(\xi)|\Big)(\xi) d\xi\\
	\leq \sum_{n\geq 1}2n \int \Big((\ast^{2n-1}|\cdot|e^{\nu t|\cdot|}|\hat{f}(\cdot)|)\ast |\cdot|e^{\nu t|\cdot|}|\hat{\Omega}(\xi)| \ast |\cdot|^{2}e^{\nu t|\cdot|}|\hat{f}(\xi)|\Big)(\xi) d\xi \\+ \sum_{n\geq 1} \int \Big((\ast^{2n}|\cdot|e^{\nu t|\cdot|}|\hat{f}(\cdot)|)\ast |\cdot|^{2}e^{\nu t|\cdot|}|\hat{\Omega}(\xi)|\Big)(\xi) d\xi\\
	\leq \sum_{n\geq 1}2n\|f\|_{\dot{\mathcal{F}}^{1,1}_\nu}^{2n-1}\|\Omega\|_{\dot{\mathcal{F}}^{1,1}_\nu}\|f\|_{\dot{\mathcal{F}}^{2,1}_\nu} + \sum_{n\geq 1} \|f\|_{\dot{\mathcal{F}}^{1,1}_\nu}^{2n}\|\Omega\|_{\dot{\mathcal{F}}^{2,1}_\nu}\\
	\leq 2A_\rho B_{1}\sum_{n\geq 1}2n\|f\|_{\dot{\mathcal{F}}^{1,1}_\nu}^{2n}\|f\|_{\dot{\mathcal{F}}^{2,1}_\nu} + 2A_\rho B_{2}\sum_{n\geq 1} \|f\|_{\dot{\mathcal{F}}^{1,1}_\nu}^{2n}\|f\|_{\dot{\mathcal{F}}^{2,1}_\nu}
	\end{multline*}
	Similarly,
	\begin{multline*}\int\! |\xi| e^{\nu t|\xi|} |\widehat{N_{3}(f)}(\xi)| d\xi\!\\ \leq\! \frac{A_{\mu}}{2}\!\Big(\!\sum_{n\geq 1}2n\|f\|_{\dot{\mathcal{F}}^{1,1}_\nu}^{2n-1}\!\|\mathcal{D}(\Omega)\|_{\dot{\mathcal{F}}^{1,1}_\nu}\!\|f\|_{\dot{\mathcal{F}}^{2,1}_\nu} \!+\! \sum_{n\geq 1} \|f\|_{\dot{\mathcal{F}}^{1,1}_\nu}^{2n}\!\|\mathcal{D}(\Omega)\|_{\dot{\mathcal{F}}^{2,1}_\nu}\!\Big)\!
	\end{multline*}
	for the $N_{3}$ nonlinear term. Plugging in the estimates \eqref{om3est} and \eqref{f11omega3} for $\mathcal{D}(\Omega)$, we obtain
	$$
	\int\! |\xi| e^{\nu t|\xi|} |\widehat{N_{3}(f)}(\xi)| d\xi\! \leq\!A_{\mu}A_{\rho}\Big(\!12C_{2}\!\sum_{n\geq 1}n\|f\|_{\dot{\mathcal{F}}^{1,1}_\nu}^{2n+1}\|f\|_{\dot{\mathcal{F}}^{2,1}_\nu} \!+\! 2C_{5}\!\sum_{n\geq 1} \|f\|_{\dot{\mathcal{F}}^{1,1}_\nu}^{2n+1}\|f\|_{\dot{\mathcal{F}}^{2,1}_\nu}\!\Big).
	$$
	By collecting the previous estimates, we obtain that
	\begin{equation}
	\label{normdecrease}
	\frac{d}{dt}\|f\|_{\dot{\mathcal{F}}^{1,1}_{\nu}} (t) \leq -\sigma\|f\|_{\dot{\mathcal{F}}^{2,1}_{\nu}},
	\end{equation}
	where	
	\begin{multline}\label{Cseriesform}
	\sigma = -\nu + A_{\rho} - 2A_{\rho}A_{\mu}C_{5}\|f\|_{\dot{\mathcal{F}}^{1,1}_{\nu}} - 2A_\rho B_{1} \sum_{n\geq 1}2n\|f\|_{\dot{\mathcal{F}}^{1,1}_{\nu}}^{2n} - 2A_\rho B_{2}\sum_{n\geq 1} \|f\|_{\dot{\mathcal{F}}^{1,1}_{\nu}}^{2n}\\ - A_{\mu}A_{\rho}\Big(\!12C_{2}\!\sum_{n\geq 1}n\|f\|_{\dot{\mathcal{F}}^{1,1}_{\nu}}^{2n+1}+\! 2C_{5}\!\sum_{n\geq 1} \|f\|_{\dot{\mathcal{F}}^{1,1}_{\nu}}^{2n+1}\Big).
	\end{multline}
	Writing the sums in a definite form,
	\begin{multline}\label{C}
	\sigma = -\nu + A_{\rho} - 2A_{\rho}A_{\mu}C_{5}\|f\|_{\dot{\mathcal{F}}^{1,1}_{\nu}} - 2A_\rho \Big(\frac{2B_{1}+B_{2}-B_{2}\|f\|_{\dot{\mathcal{F}}^{1,1}_{\nu}}^{2}}{(1-\|f\|_{\dot{\mathcal{F}}^{1,1}_{\nu}}^{2})^{2}}\Big)\|f\|_{\dot{\mathcal{F}}^{1,1}_{\nu}}^{2}\\- A_{\mu}A_{\rho}\Big(\frac{12C_{2}+2C_{5}-2C_{5}\|f\|_{\dot{\mathcal{F}}^{1,1}_{\nu}}^{2}}{(1-\|f\|_{\dot{\mathcal{F}}^{1,1}_{\nu}}^{2})^{2}}\Big)\|f\|_{\dot{\mathcal{F}}^{1,1}_{\nu}}^{3}.
	\end{multline}
	This completes the proof.
\end{proof}

\begin{remark}
We would also now like to comment on our estimate in the case of no viscosity jump, which is the regime considered in \cite{CCGRS16}. Setting $A_{\mu} = 0$, we obtain from \eqref{Cseriesform} that
$$\sigma =  A_{\rho}\left(1 -2\sum_{n\geq 1} (2n+1)\|f_0\|_{\mathcal{\dot{F}}^{1,1}}^{2n}\right).$$
Hence, $\sigma$ is a positive constant for $\|f_0\|_{\mathcal{\dot{F}}^{1,1}}$ satisfying
$$2\sum_{n\geq 1} (2n+1)\|f_0\|_{\mathcal{\dot{F}}^{1,1}}^{2n} < 1.$$
This is the condition for the 2D case in \cite{CCGRS16}. However, here we show that this condition is also sufficient in the 3D case, thereby improving the previous results.
\end{remark}

\section{$L^2$ maximum principle}\label{L2max}
For completeness, we present the proof of a $L^2$ maximum principle in this section for Muskat solutions in the viscosity jump regime. Given that the viscosities and densities of both fluids are constant on each domain \eqref{patchsolution}, from Darcy's law \eqref{Darcy} one obtains that the flow is irrotational away from the free boundary:
\begin{equation*}
\textup{curl } u(x,t)=0, \hspace{1cm} x\in D^1(t)\cup D^2(t).
\end{equation*}
Thus we find that the velocity comes from a potential $\phi$
\begin{equation}\label{potential}
u=\nabla \phi,
\end{equation}
and since the flow is incompressible we obtain that
$$\Delta \phi=0.$$
Now, integration by parts shows that
\begin{equation*}
0=\mu^i \int_{D^i}  \phi\Delta \phi \hspace{0.05cm} dx=-\mu^i \int_{D^i} \nabla \phi\cdot\nabla\phi\hspace{0.05cm} dx+\mu^i \int_{\partial D^i}\nabla\phi \cdot n\phi \hspace{0.05cm} d\sigma,
\end{equation*}
so using \eqref{potential} it reads as 
\begin{equation*}
-\mu^i\int_{D^i}|u|^2\hspace{0.05cm} dx+\int_{\partial D^i}u\cdot n \mu^i \phi \hspace{0.05cm}d\sigma =0.
\end{equation*}
Recalling that the normal velocity is continuous across the boundary due to the incompressibility condition, by adding the balance of both domains we can write
\begin{equation}
-\int_{\mathbb{R}^3}\mu |u|^2 \hspace{0.05cm}dx+\int_{\partial D} u\cdot n(\mu^2\phi^2-\mu^1\phi^1)\hspace{0.05cm}d\sigma=0.
\end{equation} 
Here $\phi^i$ is the potential in $D^i$.
Introducing \eqref{potential} in \eqref{Darcy} we find that
\begin{equation*}
\mu^i \phi^i=-p-\rho^i x_3.
\end{equation*}
From this and the continuity of the pressure along the boundary we obtain that
\begin{equation}\label{aux}
-\int_{\mathbb{R}^3} \mu |u|^2\hspace{0.05cm}dx=(\rho_2-\rho_1)\int_{\partial D}u\cdot n x_3\hspace{0.05cm} d\sigma.
\end{equation}
If the boundary is described as a graph
$$\partial D(t)=\{(\al,f(\al,t))\in \mathbb{R}^3:\al\in \mathbb{R}^3\},$$
since it moves with the flow one has that
\begin{equation*}
\begin{aligned}
f_t(\alpha)&=u(\alpha,f(\alpha))\cdot (-\partial_{\alpha_1}f(\alpha),-\partial_{\alpha_2}f(\alpha),1))\\
&=u(\alpha,f(\alpha))\cdot n(\alpha) \sqrt{1+(\partial_{\alpha_1}f(\alpha))^2+(\partial_{\alpha_2}f(\alpha))^2}.
\end{aligned}
\end{equation*}
Going back to \eqref{aux} we find that
\begin{equation*}
(\rho_2-\rho_1)\int_{\mathbb{R}^2}f_t(\al)f(\al)\, d\al+\int_{\mathbb{R}^3} \mu |u|^2\hspace{0.05cm}dx=0,
\end{equation*}
so by integration in time we finally obtain the $L^2$ maximum principle
\begin{equation*}
(\rho^2-\rho^1)\|f\|_{L^2}^2(t)+2\int_{\mathbb{R}^3} \mu |u|^2\hspace{0.05cm}dx=(\rho_2-\rho_1)\|f_0\|_{L^2}^2.
\end{equation*}

\section{Uniqueness}\label{uniquenesssection}

\begin{prop}\label{uniqueness}
	Consider two solutions $f$ and $g$ to the Muskat problem with initial data $f_{0},g_0\in L^{2} \cap \mathcal{\dot{F}}^{1,1}$ that satisfy the condition \eqref{condition}. Then, $$\frac{d}{dt}\|f-g\|_{\mathcal{F}^{0,1}} \leq -C \|f-g\|_{\mathcal{\dot{F}}^{1,1}},$$ and moreover, $\|f-g\|_{L^{\infty}} = 0.$
\end{prop}
\begin{proof}
	Using \eqref{decompnonlinear}, we can write, as before
	\begin{align*}
	\frac{d}{dt}\|f-g\|_{\mathcal{F}^{0,1}_{\nu}} &= \int_{\mathbb{R}^{2}}  \frac{1}{2}\frac{\overline{\widehat{(f-g)}}(\xi)\partial_{t}(\widehat{f-g})(\xi) + \widehat{(f-g)}(\xi)\partial_{t}(\overline{\widehat{f-g}})(\xi)}{|\widehat{(f-g)}(\xi)|}\\
	&\leq (\nu-A_{\rho})\|f-g\|_{\mathcal{\dot{F}}^{1,1}} + \sum_{i=1}^{3} \int |\widehat{N_{i}(f-g)}(\xi)|d\xi
	\end{align*}
	where $N_{i}$ are the nonlinear terms given by \eqref{N1N2N3}. For example,
	$$\int |\widehat{N_{1}(f-g)}(\xi)|d\xi = \frac{A_\mu}{2} \int_{\mathbb{R}^{2}} |\widehat{\Lambda D(\Omega(f))}(\xi) - \widehat{\Lambda D(\Omega(g))}(\xi)| d\xi,$$ where $\Omega(f)$ is the term $\Omega$ in the case of the solution $f$ and similarly for $\Omega(g)$. We define the terms $N_{2}$ and $N_{3}$ later. As earlier in the paper, we use the decomposition
	\eqref{riesz}, where $\partial_{\alpha_{i}}\mathcal{D}(\Omega)$ is given by \eqref{partialD}. Hence, we can write for $i=1$
	\begin{multline*}
	\int_{\mathbb{R}^{2}} |\widehat{\partial_{\alpha_{1}} D(\Omega(f))}(\xi) - \widehat{\partial_{\alpha_{1}} D(\Omega(g))}(\xi)| d\xi \\ \leq \int_{\mathbb{R}^{2}} |\widehat{BR_{1}(f)}(\xi) - \widehat{BR_{1}(g)}(\xi)| d\xi + \int_{\mathbb{R}^{2}} |\widehat{BR_{3}(f)\partial_{\alpha_{1}}f}(\xi) - \widehat{BR_{3}(g)\partial_{\alpha_{1}}g}(\xi)| d\xi.
	\end{multline*}
	First, we consider the $BR_{1} = BR_{11} + BR_{12}$ term. Using the Taylor expansion, we can write $BR_{11}(f)$ as 
	\begin{equation*}
	\begin{aligned}
	BR_{11}(f) &= -\frac{1}{8\pi}\sum_{n \geq 0} \text{p.v.} \int_{\mathbb{R}^{2}}  \Big(\omega_{3}(f)(\alpha-\beta)  (-1)^{n} a_{n} (\Delta_{\beta}f(\alpha))^{2n}  \\
	&\quad-  \omega_{3}(f)(\alpha+\beta) (-1)^{n} a_{n} (\Delta_{-\beta}f(\alpha))^{2n}\Big) \frac{\beta_{2}d\beta}{|\beta|^{3}} \\ 
	&\eqdef BR_{11}^{+}(f) - BR_{11}^{-}(f).
	\end{aligned}
	\end{equation*}
	Next, we get that the integrand of the n-th term in $BR_{11}(f) - BR_{11}(g)$ is given by
	\begin{multline}\label{addsubunique}
	BR_{11}^{+}(f) - BR_{11}^{+}(g)  = (-1)^{n} a_{n}(p_{1}p_{2}^{2n} - q_{1}q_{2}^{2n}) \\ = (-1)^{n} a_{n} (p_{1}p_{2}^{2n} - q_{1}p_{2}^{2n} + q_{1}p_{2}^{2n} - q_{1}q_{2}p_{2}^{2n-1} + q_{1}q_{2}p_{2}^{2n-1} - \ldots + q_{1}q_{2}^{2n-1}p_{2} -  q_{1}q_{2}^{2n})\\
	=(-1)^{n} a_{n} \Big( (p_{1}- q_{1})p_{2}^{2n} + q_{1}p_{2}^{2n-1}(p_{2}-q_{2}) + q_{1}q_{2}p_{2}^{2n-2}(p_{2}-q_{2}) +\\ \ldots +  q_{1}q_{2}^{2n-1}(p_{2}-q_{2})\Big)
	\end{multline}
	where $p_{1} = \omega_{3}(f)(\alpha-\beta)$, $p_{2} = \Delta_{\beta}(f)(\alpha)$, $q_{1} = \omega_{3}(g)(\alpha-\beta)$ and $q_{2} = \Delta_{\beta}(g)(\alpha)$.
	We do the same for $BR_{11}^{-}$ by defining $p_{-1} = \omega_{3}(f)(\alpha+\beta)$, $p_{-2} = \Delta_{-\beta}(f)(\alpha)$, $q_{-1} = \omega_{3}(g)(\alpha+\beta)$ and $q_{-2} = \Delta_{-\beta}(g)(\alpha)$. Next, using the Fourier arguments to bound $BR_{11}$ as in Section \ref{secw}, we can obtain that
	\begin{multline}\label{uniquebr11}
	|\widehat{BR_{11}(f)}(\xi)-\widehat{BR_{11}(g)}(\xi)|\\ \leq \frac{1}{2} \sum_{n\geq 0} |\widehat{\omega_{3}(f)-\omega_{3}(g)}|\ast (\ast^{2n}|\cdot||\hat{f}|) + |\widehat{\omega_{3}(g)}|\ast (\ast^{2n-1}|\cdot||\hat{f}|)\ast |\cdot||\widehat{f-g}(\cdot)| \\  + \ldots + |\widehat{\omega_{3}(g)}|\ast (\ast^{2n-1}|\cdot||\hat{g}|)\ast |\cdot||\widehat{f-g}(\cdot)|.
	\end{multline}
	Hence, applying Young's inequality,
	\begin{multline}\label{br11unique}
	\int_{\mathbb{R}^{2}}|\widehat{BR_{11}(f)}(\xi)-\widehat{BR_{11}(g)}(\xi)| d\xi \\ \leq \frac{1}{2} \sum_{n\geq 0} \|\omega_{3}(f)-\omega_{3}(g)\|_{\mathcal{F}^{0,1}}\|f\|_{\mathcal{\dot{F}}^{1,1}}^{2n} +  \|\omega_{3}(g)\|_{\mathcal{F}^{0,1}}\|f\|_{\mathcal{\dot{F}}^{1,1}}^{2n-1}\|f-g\|_{\mathcal{\dot{F}}^{1,1}}\\ + \ldots +  \|\omega_{3}(g)\|_{\mathcal{F}^{0,1}}\|g\|_{\mathcal{\dot{F}}^{1,1}}^{2n-1}\|f-g\|_{\mathcal{\dot{F}}^{1,1}}.
	\end{multline}
	Next,
	\begin{multline*}
	\omega_{3}(f) -\omega_{3}(g)\\ = \partial_{\alpha_{2}}D(\Omega(f))\partial_{\alpha_{1}}f - \partial_{\alpha_{2}}D(\Omega(g))\partial_{\alpha_{1}}g - \partial_{\alpha_{1}}D(\Omega(f))\partial_{\alpha_{2}}f + \partial_{\alpha_{1}}D(\Omega(g))\partial_{\alpha_{2}}g
	\\=\partial_{\alpha_{2}}(D(\Omega(f))-D(\Omega(g)))\partial_{\alpha_{1}}f + \partial_{\alpha_{2}}D(\Omega(g))(\partial_{\alpha_{1}}(f-g))\\- \partial_{\alpha_{1}}(D(\Omega(f))-D(\Omega(g)))\partial_{\alpha_{2}}f - \partial_{\alpha_{1}}D(\Omega(g))(\partial_{\alpha_{2}}(f-g)).
	\end{multline*}
	Hence,
	\begin{multline}\label{omega3unique}
	\|\omega_{3}(f) -\omega_{3}(g)\|_{\mathcal{F}^{0,1}}\\ \leq \|\partial_{\alpha_{1}}(D(\Omega(f))-D(\Omega(g)))\|_{\mathcal{F}^{0,1}}\|f\|_{\mathcal{\dot{F}}^{1,1}} + \|\partial_{\alpha_{2}}(D(\Omega(f)-\Omega(g)))\|_{\mathcal{F}^{0,1}}\|f\|_{\mathcal{\dot{F}}^{1,1}}\\ + 
	\|\partial_{\alpha_{1}}D(\Omega(g))\|_{\mathcal{F}^{0,1}}\|f-g\|_{\mathcal{\dot{F}}^{1,1}} + \|\partial_{\alpha_{2}}D(\Omega(g))\|_{\mathcal{F}^{0,1}}\|f-g\|_{\mathcal{\dot{F}}^{1,1}}.
	\end{multline}
	Furthermore, for $BR_{12}$ we similarly obtain
	\begin{multline*}
	\|BR_{12}(f)-BR_{12}\|_{\mathcal{F}^{0,1}} \\\leq \frac{1}{2}\sum_{n\geq 0}  \|\partial_{\alpha_{1}}\Omega(g)-\partial_{\alpha_{1}}\Omega(g)\|_{\dot{\mathcal{F}}^{0,1}}\|f\|_{\mathcal{\dot{F}}^{1,1}}^{2n+1} +  \|\partial_{\alpha_{1}}\Omega(g)\|_{\mathcal{F}^{0,1}}\|f\|_{\mathcal{\dot{F}}^{1,1}}^{2n}\|f-g\|_{\mathcal{\dot{F}}^{1,1}}\\ + \ldots +  \|\partial_{\alpha_{1}}\Omega(g)\|_{\mathcal{F}^{0,1}}\|g\|_{\mathcal{\dot{F}}^{1,1}}^{2n}\|f-g\|_{\dot{\mathcal{F}}^{1,1}}.
	\end{multline*}
	Next, for the $BR_{3}$ integral term
	\begin{multline*}
	\int_{\mathbb{R}^{2}} |\widehat{BR_{3}(f)\partial_{\alpha_{1}}f}(\xi) - \widehat{BR_{3}(g)\partial_{\alpha_{1}}g}(\xi)| d\xi\\ \leq \|BR_{3}(f)-BR_{3}(g)\|_{\mathcal{F}^{0,1}}\|f\|_{\mathcal{\dot{F}}^{1,1}} + \|BR_{3}(g)\|_{\dot{\mathcal{F}^{0,1}}}\|f-g\|_{\dot{\mathcal{F}}^{1,1}}.
	\end{multline*}
	Next,
	\begin{multline*}
	\|BR_{3}(f)-BR_{3}(g)\|_{\mathcal{F}^{0,1}} \\ \leq \frac{1}{2}\sum_{i=1,2}\sum_{n\geq 0}  \|\partial_{\alpha_{i}}\Omega(f)-\partial_{\alpha_{i}}\Omega(g)\|_{\mathcal{F}^{0,1}}\|f\|_{\mathcal{\dot{F}}^{1,1}}^{2n} +  \|\partial_{\alpha_{i}}\Omega(g)\|_{\mathcal{F}^{0,1}}\|f\|_{\mathcal{\dot{F}}^{1,1}}^{2n-1}\|f-g\|_{\mathcal{\dot{F}}^{1,1}}\\ + \ldots +  \|\partial_{\alpha_{i}}\Omega(g)\|_{\mathcal{F}^{0,1}}\|g\|_{\mathcal{\dot{F}}^{1,1}}^{2n-1}\|f-g\|_{\dot{\mathcal{F}}^{1,1}}.
	\end{multline*}
	The key point to note here is that these estimates are precisely those used to prove the vorticity estimates in the norm $\dot{\mathcal{F}}^{1,1}$ in Proposition \ref{f11vorticity} if we replace the quantities
	\begin{equation}\label{swap}
	\begin{aligned}
	\|\partial_{\alpha_{i}}\Omega(f)\|_{\dot{\mathcal{F}}^{1,1}} \text{   or    } \|\partial_{\alpha_{i}}\Omega(g)\|_{\dot{\mathcal{F}}^{1,1}}&\leftrightarrow \|\partial_{\alpha_{i}}\Omega(g)-\partial_{\alpha_{i}}\Omega(g)\|_{\mathcal{F}^{0,1}}\\
	\|f\|_{\dot{\mathcal{F}}^{2,1}} \text{   or    } \|g\|_{\dot{\mathcal{F}}^{2,1}}&\leftrightarrow \|f-g\|_{\dot{\mathcal{F}}^{1,1}},
	\end{aligned}
	\end{equation}
	and notice by counting terms that the computation of \eqref{addsubunique} creates the same effect on the estimates as the effect created by the triangle inequality (or product rule) in the case of estimates of Proposition \ref{f11vorticity}. Therefore, continuing to compute the estimates for uniqueness as above and comparing with the estimates of Section \ref{secw} and \ref{secanalytic} by using the substitutions \eqref{swap}, we obtain the analogous estimate, for example:
	$$\|\partial_{\alpha_{i}}\Omega(f)-\partial_{\alpha_{i}}\Omega(g)\|_{\mathcal{F}^{0,1}} \leq 2A_{\rho}C_{4}\|f-g\|_{\dot{\mathcal{F}}^{1,1}}.$$
	These vorticity estimates and performing similar computations on the nonlinear terms $N_{i}$, we can see that
	$$\frac{d}{dt}\|f-g\|_{\mathcal{F}^{0,1}} \leq -\sigma \|f-g\|_{\dot{\mathcal{F}}^{1,1}}$$
	where $\sigma$ is the same positive constant as in Proposition \ref{normdecreaseprop}. It can be seen by the swap of terms described above in \eqref{swap}.\end{proof}

\section{Regularization}\label{regularization}

In this section, we describe the regularization of the system together with the limit process to get bona-fide and not just a priori estimates for the Muskat problem. We denote the heat kernel $\zeta_\varepsilon$ as an approximation to the identity where $\varepsilon$ plays the role of time in such a way that $\zeta_\varepsilon$ converges to the identity as $\varepsilon\to 0^+$. We consider the following regularization of the system
\begin{equation}\label{Epsiloncontourequation}
\partial_tf^\varepsilon=-\frac{A_\rho}2\Lambda(\zeta_\varepsilon*\zeta_\varepsilon*f^\varepsilon)+\zeta_\varepsilon*(N(\zeta_\varepsilon*\zeta_\varepsilon*f^\varepsilon,\Omega^\varepsilon)),\quad f^\varepsilon(x,0)=(\zeta_\varepsilon*f_0)(x).
\end{equation}
where $N(\cdot,\cdot)$ is given by \eqref{decompnonlinear} and \eqref{N1N2N3}, and $\Omega^\varepsilon$ by
\begin{equation}\label{EpsilonOmegagraph}
\Omega^\varepsilon(\alpha,t)=A_\mu \mathcal{D}^\varepsilon(\Omega^\varepsilon)(\alpha,t)
-2A_\rho\zeta_\varepsilon*\zeta_\varepsilon* f^\varepsilon(\alpha,t).
\end{equation}
The operator $\mathcal{D}^\varepsilon(\Omega^\varepsilon)$ is written as follows
\begin{equation}\label{EpsilonDOmegagraph}
\mathcal{D}^\varepsilon(\Omega^\varepsilon)(\alpha)\! =\! \frac{1}{2\pi}\!\! \int_{\mathbb{R}^{2}}\!\!\! \frac{\frac{\beta}{|\beta|}\!\cdot\!\nabla_{\alpha}(\zeta_\varepsilon\!*\!\zeta_\varepsilon\!*\! f^\varepsilon)(\alpha\!-\!\beta)\!-\!\Delta_{\beta}(\zeta_\varepsilon\!*\zeta_\varepsilon\!*\!f^ \varepsilon)(\alpha)}{(1+(\Delta_{\beta}(\zeta_\varepsilon\!*\!\zeta_\varepsilon\!*\!f^\varepsilon)(\alpha))^{2})^{3/2}}\frac{\Omega^\varepsilon(\alpha\!-\!\beta)}{|\beta|^{2}}d\beta.
\end{equation}
Integration by parts also provides the identities
\begin{align}
\begin{split}
\partial_{\alpha_i}\mathcal{D}^\varepsilon(\Omega^\varepsilon)&=-\frac{1}{2\pi}\! \int_{\mathbb{R}^{2}}\!\! \frac{\Delta_{\beta}(\partial_{\alpha_i}(\zeta_\varepsilon\!*\zeta_\varepsilon\!*\!f^ \varepsilon))(\alpha)}{(1+(\Delta_{\beta}(\zeta_\varepsilon\!*\!\zeta_\varepsilon\!*\!f^\varepsilon)(\alpha))^{2})^{3/2}}\frac{\beta\cdot\nabla_{\alpha}\Omega^\varepsilon(\alpha\!-\!\beta)}{|\beta|^{2}}d\beta \\
+\frac{1}{2\pi}\! &\int_{\mathbb{R}^{2}}\!\! \frac{\frac{\beta}{|\beta|}\cdot\nabla_{\alpha}(\zeta_\varepsilon\!*\!\zeta_\varepsilon\!*\! f^\varepsilon)(\alpha\!-\!\beta)\!-\!\Delta_{\beta}(\zeta_\varepsilon\!*\zeta_\varepsilon\!*\!
	f^\varepsilon)(\alpha)}{(1+(\Delta_{\beta}(\zeta_\varepsilon\!*\!\zeta_\varepsilon\!*\!f^\varepsilon)(\alpha))^{2})^{3/2}}\frac{\partial_{\alpha_i}\Omega^\varepsilon(\alpha\!-\!\beta)}{|\beta|^{2}}d\beta.
\end{split}
\end{align}

Then it is easy to estimate $\Omega^\epsilon$ as in Section \ref{secw} in terms of $\zeta_\varepsilon\!*\!\zeta_\varepsilon\!*\!f^\varepsilon$ with the condition $\|\zeta_\varepsilon*\zeta_\varepsilon*f^\varepsilon\|_{\dot{\mathcal{F}}^{1,1}}(t)<1$. These estimates provide a local existence result using the classical Picard theorem on the Banach space $C([0,T_\varepsilon];H^4)$. We find the abstract evolution system given by $\partial_tf^\varepsilon=G(f^\varepsilon)$ where $G$ is Lipschitz on the open set $\{g(x)\in H^4:\|g\|_{\dot{\mathcal{F}}^{1,1}}<1\}$. We remember that $f^\varepsilon(x,0)\in H^4$ due to $f_0\in L^2$.   The next step is to reproduce estimate \eqref{normdecrease} for $s=1$. As the convolutions are taken with the heat kernel, it is easy to prove analyticity for $f^\varepsilon$ so that for $\nu$ small enough we find that $\|f^\epsilon\|_{\dot{\mathcal{F}}^{1,1}_\nu}$ bounded. Even more, we know that $\|f^\epsilon\|_{\dot{\mathcal{F}}^{1,1}_\nu}(t)<k(|A_\mu|)$, as continuity in time provides that this quantity is close in size to 
$\|f^\epsilon\|_{\dot{\mathcal{F}}^{1,1}_\nu}(0)=\|f^\epsilon\|_{\dot{\mathcal{F}}^{1,1}}(0)\leq \|f_0\|_{\dot{\mathcal{F}}^{1,1}}<k(|A_\mu|)$ if $T_\varepsilon>0$ is small enough. Therefore, in checking its evolution as in Section \ref{secanalytic} we find that
$$
\frac{d}{dt}\|f^\epsilon\|_{\dot{\mathcal{F}}^{1,1}_\nu}
(t)\leq -C\|\zeta_\varepsilon\!*\zeta_\varepsilon\!*\!
f^\varepsilon\|_{\dot{\mathcal{F}}^{2,1}_\nu},
$$
so that integration in time provides
\begin{equation}\label{epsineqfun}
\|f^\varepsilon\|_{\dot{\mathcal{F}}^{1,1}_\nu}
(t)+C\int_0^t\|\zeta_\varepsilon\!*\zeta_\varepsilon\!*\!
f^\varepsilon\|_{\dot{\mathcal{F}}^{2,1}_\nu}(\tau)d\tau\leq \|f_0\|_{\dot{\mathcal{F}}^{1,1}}.
\end{equation}
Next we repeat the computations in Section \ref{L2section} for the regularized system. It is possible to find that

$$
\|f^\varepsilon\|_{L^{2}_{\nu}}(t)\leq \|f_{0}\|_{L^{2}}\exp\Big(R(\|f_0\|_{\dot{\mathcal{F}}^{1,1}})\Big).$$
Energy estimates provide 
$$
\frac{d}{dt}\|f^\varepsilon\|^2_{H^4}\leq P(\|\zeta_\varepsilon\!*\!
f^\varepsilon\|^2_{H^4})
$$
where $P$ is a polynomial function. Then, using that $$\|\zeta_\varepsilon\!*\!
f^\varepsilon\|_{H^4}\leq C(\varepsilon)\|f^\varepsilon\|_{L^2}\leq C(\varepsilon)\|f_{0}\|_{L^{2}}\exp\Big(C(\|f_0\|_{\dot{\mathcal{F}}^{1,1}})\Big)
$$
we are able to extend the solutions in $C([0,T];H^4)$ for any $T>0$.

Next, we find a candidate for a solution by taking the limit $\varepsilon\to 0^+$ after proving that $f^{\epsilon}$ is Cauchy $L^\infty(0,T;\mathcal{F}^{0,1})$. From now on, we consider $\varepsilon\geq \varepsilon'>0$. Then, as in Section \ref{uniquenesssection}, we are able to find that
\begin{align*}
\|f^\varepsilon-f^{\varepsilon'}\|_{\mathcal{F}^{0,1}}(t)\leq & \|\zeta_\varepsilon*f_0-\zeta_{\varepsilon'}*f_0\|_{\mathcal{F}^{0,1}}+I_1(t)+I_2(t)
\end{align*}
where
$$
I_1(t)=\int_0^t\frac{A_\rho}2\|\Lambda(\zeta_\varepsilon*\zeta_\varepsilon*f^{\varepsilon'}-\zeta_{\varepsilon'}*\zeta_{\varepsilon'}*f^{\varepsilon'})\|_{\mathcal{F}^{0,1}}(\tau)d\tau87,
$$
and
$$
I_2(t)=\int_0^t\|\zeta_\varepsilon*N(\zeta_\varepsilon*\zeta_\varepsilon*f^{\varepsilon'},\Omega^{\varepsilon'})-\zeta_{\varepsilon'}*N(\zeta_{\varepsilon'}*\zeta_{\varepsilon'}*f^{\varepsilon'},\Omega^{\varepsilon'})\|_{\mathcal{F}^{0,1}}(s)ds.
$$
As before, in order to get the inequality above, we use the decay from the dissipation term to absorb the bounds for $\zeta_\varepsilon*N(\zeta_\varepsilon*\zeta_\varepsilon*f^{\varepsilon},\Omega^{\varepsilon})-\zeta_{\varepsilon}*N(\zeta_{\varepsilon}*\zeta_{\varepsilon}*f^{\varepsilon'},\Omega^{\varepsilon'})$. 
Then, using the mean value theorem in the heat kernel on the Fourier side, it is possible to get
\begin{equation}\label{geteps}
\|\zeta_\varepsilon*f_0-\zeta_{\varepsilon'}*f_0\|_{\mathcal{F}^{0,1}}\leq C\|f_0\|_{\dot{\mathcal{F}}^{1,1}}\varepsilon^{1/2}.
\end{equation}

Similarly 
$$
I_1(t)\leq C \int_0^t\|\zeta_{\varepsilon'}*\zeta_{\varepsilon'}*f^{\varepsilon'}\|_{\mathcal{F}^{2,1}}(s)ds \,\varepsilon^{1/2} \leq C\|f_0\|_{\dot{\mathcal{F}}^{1,1}}\varepsilon^{1/2}.
$$
A further splitting in the mollifiers, together with the inequality
$$
\|\zeta_{\varepsilon}*\zeta_{\varepsilon}*f^{\varepsilon'}\|_{\mathcal{F}^{s,1}}(s)\leq 
\|\zeta_{\varepsilon'}*\zeta_{\varepsilon'}*f^{\varepsilon'}\|_{\mathcal{F}^{s,1}}(s), \quad s\geq 0,
$$ 
allows us to find for the nonlinear term, as before, the following bound
$$
I_2(t)\leq C(\|f_0\|_{\dot{\mathcal{F}}^{1,1}}) \int_0^t\|\zeta_{\varepsilon'}*\zeta_{\varepsilon'}*f^{\varepsilon'}\|_{\mathcal{F}^{2,1}}(s)ds \,\varepsilon^{1/2} \leq C(\|f_0\|_{\dot{\mathcal{F}}^{1,1}})\varepsilon^{1/2}.
$$
It yields finally
\begin{equation}\label{cauchy}
\|f^\varepsilon-f^{\varepsilon'}\|_{\mathcal{F}^{0,1}}(t)\leq C(\|f_0\|_{\dot{\mathcal{F}}^{1,1}})\varepsilon^{1/2},
\end{equation}
so that we are done finding a limit $f\in L^\infty(0,T;\mathcal{F}^{0,1})$. The interpolation inequality 
$$
\|g\|_{\dot{\mathcal{F}}^{1,1}}^2\leq \|g\|_{\mathcal{F}^{0,1}}\|g\|_{\dot{\mathcal{F}}^{2,1}}
$$
provides
\begin{align}
\begin{split}\label{solutionconverges}
\int_0^t\|\zeta_{\varepsilon}*\zeta_{\varepsilon}*f^{\varepsilon}-&\zeta_{\varepsilon'}*\zeta_{\varepsilon'}*f^{\varepsilon'}\|_{\dot{\mathcal{F}}^{1,1}}^2(s)ds\leq \\
&\int_0^tA(s)(\|\zeta_{\varepsilon}*\zeta_{\varepsilon}*f^{\varepsilon}\|_{\dot{\mathcal{F}}^{2,1}}(s)+\|\zeta_{\varepsilon'}*\zeta_{\varepsilon'}*f^{\varepsilon'}\|_{\dot{\mathcal{F}}^{2,1}}(s))ds,
\end{split}
\end{align}
where
$$
A(s)=\|\zeta_{\varepsilon}*\zeta_{\varepsilon}*(f^{\varepsilon}-f^{\varepsilon'})\|_{\mathcal{F}^{0,1}}(s)+\|\zeta_{\varepsilon}*\zeta_{\varepsilon}*f^{\varepsilon'}-\zeta_{\varepsilon'}*\zeta_{\varepsilon'}*f^{\varepsilon'}\|_{\mathcal{F}^{0,1}}(s).
$$
The first term in $A(s)$ is controlled by \eqref{cauchy} and for the second term we apply a similar approach as in \eqref{geteps} to find   
$$
A(s)\leq C(\|f_0\|_{\dot{\mathcal{F}}^{1,1}})\varepsilon^{1/2}.
$$
Using \eqref{epsineqfun} in \eqref{solutionconverges} we find finally
\begin{align}
\begin{split}\label{solutionconvergesStrong}
\int_0^t\|\zeta_{\varepsilon}*\zeta_{\varepsilon}*f^{\varepsilon}-\zeta_{\varepsilon'}*\zeta_{\varepsilon'}*f^{\varepsilon'}\|_{\dot{\mathcal{F}}^{1,1}}^2(s)ds\leq C(\|f_0\|_{\dot{\mathcal{F}}^{1,1}})\varepsilon^{1/2},
\end{split}
\end{align}
which provides strong convergence of $\zeta_{\varepsilon}*\zeta_{\varepsilon}*f^{\varepsilon}$ to $f$ in $L^2(0,T;\dot{\mathcal{F}}^{1,1})$.

Next we can extract a subsequence $f^{\varepsilon_n}$ in such a way that 
$$
(\widehat{f}^{\varepsilon_n}(\xi,t),\exp{(-8\pi^2\varepsilon_n|\xi|^2)}\widehat{f}^{\varepsilon_n}(\xi,t))\to (\widehat{f}(\xi,t),\widehat{f}(\xi,t))
$$
pointwise for almost every $(\xi,t)\in\mathbb{R}^2\times[0,T]$. Therefore for $t\in [0,T]\smallsetminus Z$ with measure $|Z|=0$ it is possible to find the same pointwise for almost every $\xi\in\mathbb{R}^2$. Fatou's lemma allows us to conclude that for $t\in [0,T]\smallsetminus Z$
and
$$
M(t)=\|f\|_{\dot{\mathcal{F}}^{1,1}_\nu}
(t)+C\int_0^t\|f\|_{\dot{\mathcal{F}}^{2,1}_\nu}(s)
$$
it is possible to obtain
\begin{align*}
\begin{split}
M(t)\leq\liminf_{n\to\infty}\Big(\|f^{\varepsilon_n}\|_{\dot{\mathcal{F}}^{1,1}_\nu}
(t)+C\int_0^t\|\zeta_{\varepsilon_n}\!*\zeta_{\varepsilon_n}\!*\!
f^{\varepsilon_n}\|_{\dot{\mathcal{F}}^{2,1}_\nu}(s)ds\Big)\leq \|f_0\|_{\dot{\mathcal{F}}^{1,1}},
\end{split}
\end{align*}

The strong convergence of $\zeta_{\varepsilon}*\zeta_{\varepsilon}*f^{\varepsilon}$ to $f$ in $L^2(0,T;\dot{\mathcal{F}}^{1,1})$ together with the regularity found for $f$ allow us to take the limit in equations (\ref{Epsiloncontourequation},\ref{EpsilonOmegagraph},\ref{EpsilonDOmegagraph}) to find $f$ as a solution to the original Muskat equations (\ref{decompnonlinear}-\ref{DOmegagraph}). Now we use the approach in Section \ref{L2max} to get the $L^2$ maximum principle for $f$.

\section{Gain of $L^{2}$ Derivatives with Analytic Weight}\label{L2section}

In this section, we first show gain of $L^{2}$ regularity. In particular, we prove uniform bounds in $L^{2}_{\nu}=\mathcal{F}^{0,2}_{\nu}$, which will be used to show decay of analytic $L^{2}$ norms, and more prominently, the ill-posedness argument of Section \ref{illposedsection}:

\begin{thm}\label{L2}
	Suppose $f_{0}\in L^{2}\cap \dot{\mathcal{F}}^{1,1}$ and $\|f_{0}\|_{\dot{\mathcal{F}}^{1,1}}$ satisfying the condition \eqref{condition}. Then, $f(t) \in L^{2}_{\nu}$ instantly for all $t>0$. Moreover
	$$\|f\|_{L^{2}_{\nu}}^2(t) \leq \|f_{0}\|_{L^{2}}^2\exp(R(\|f_{0}\|_{\dot{\mathcal{F}}^{1,1}})),$$
	with $R$ a rational function. In particular, this implies that $f(t)\in H^{s}$ instantly for all $t>0$.
\end{thm}

\begin{proof}
	Differentiating the 
	\begin{multline*}
	\frac{1}{2}\frac{d}{dt}\|f\|_{L^{2}_{\nu}}^{2}(t)  = (\nu-A_{\rho})\|f\|_{\dot{H}^{1/2}_{\nu}}^{2} + \frac{A_{\mu}}{2}\int |\xi|e^{2\nu t |\xi|}|\widehat{\mathcal{D}(\Omega)}(\xi)||\hat{f}(\xi)|d\xi  \\+ \int e^{2\nu t |\xi|}|\widehat{N_{2}}(\xi)||\hat{f}(\xi)|d\xi + \int e^{2\nu t |\xi|}|\widehat{N_{3}}(\xi)||\hat{f}(\xi)|d\xi.
	\end{multline*}
	We now bound the nonlinear terms. For example,
	$$\int |\xi|e^{2\nu t |\xi|}|\widehat{\mathcal{D}(\Omega)}(\xi)||\hat{f}(\xi)|d\xi \leq \|f\|_{\dot{H}^{1/2}_{\nu}}\|\mathcal{D}(\Omega)\|_{\dot{H}^{1/2}_{\nu}},$$
	and using the bounds on $\hat{N}_i$ \eqref{N2est} in \eqref{N3est} followed by the product rule we obtain that
	\begin{multline*}
	\int e^{2\nu t |\xi|}|\widehat{N}_{2}(\xi)||\hat{f}(\xi)|d\xi \leq \sum_{n\geq 1} \int e^{2\nu t |\xi|}|\hat{f}(\xi)| \left(|\cdot||\hat{\Omega}(\cdot)|\ast(\ast^{2n}|\cdot||\hat{f}(\cdot)|)\right)(\xi) d\xi\\
	\leq \sum_{n\geq 1} \int e^{\nu t |\xi|}|\hat{f}(\xi)|( |\cdot||\hat{\Omega}(\cdot)|e^{\nu t |\cdot|})\ast(\ast^{2n}|\cdot|e^{\nu t |\cdot|}|\hat{f}(\cdot)|)(\xi)d\xi\\
	\leq \sum_{n\geq 1} \int |\xi||\hat{\Omega}(\xi)|e^{\nu t |\xi|}\cdot (e^{\nu t |\cdot|}|\hat{f}(\cdot)|) \ast(\ast^{2n}|\cdot|e^{\nu t |\cdot|}|\hat{f}(\cdot)|)d\xi\\
	\leq \sum_{n\geq 1} 2n  \int |\xi|^{\frac{1}{2}}|\widehat{\Omega}(\xi)|e^{\nu t |\xi|}\cdot(e^{\nu t |\cdot|}|\hat{f}(\cdot)|)\ast (|\cdot|^{\frac{3}{2}}e^{\nu t |\cdot|}|\hat{f}(\cdot)|) \ast(\ast^{2n-1}|\cdot|e^{\nu t |\cdot|}|\hat{f}(\cdot)|)d\xi \\
	+ \sum_{n\geq 1}\int |\xi|^{\frac{1}{2}}|\widehat{\Omega}(\xi)|e^{\nu t |\xi|}\cdot (|\cdot|^{\frac{1}{2}}e^{\nu t |\cdot|}|\hat{f}(\cdot)|) \ast(\ast^{2n}|\cdot|e^{\nu t |\cdot|}|\hat{f}(\cdot)|)d\xi\\
	\leq  \sum_{n\geq 1} 2n \|\Omega\|_{\dot{H}^{1/2}_{\nu}} \|f\|_{L^{2}_{\nu}}\|f\|_{\dot{\mathcal{F}}^{3/2,1}_{\nu}}\|f\|_{\dot{\mathcal{F}}^{1,1}_{\nu}}^{2n-1} + \|\Omega\|_{\dot{H}^{1/2}_{\nu}} \|f\|_{\dot{H}^{1/2}_{\nu}}\|f\|_{\dot{\mathcal{F}}^{1,1}_{\nu}}^{2n}\\
	\leq  \sum_{n\geq 1} 2n\frac{\epsilon_{n}}{2}\|\Omega\|_{\dot{H}^{1/2}_{\nu}}^{2} +2n \frac{1}{2\epsilon_{n}}\|f\|_{L^{2}_{\nu}}^{2}\|f\|_{\dot{\mathcal{F}}^{3/2,1}_{\nu}}^{2}\|f\|_{\dot{\mathcal{F}}^{1,1}_{\nu}}^{4n-2} \\+ \|\Omega\|_{\dot{H}^{1/2}_{\nu}} \|f\|_{\dot{H}^{1/2}_{\nu}}\|f\|_{\dot{\mathcal{F}}^{1,1}_{\nu}}^{2n},
	\end{multline*}
	where the last line is obtained using Young's inequality for products. We set $\epsilon_{n} = \epsilon/n^{3}$ for some small constant $\epsilon > 0$ that we can pick.
	We can bound the other terms of $N_{2}$ and $N_{3}$ similarly.
	It remains to bound $\|\mathcal{D}(\Omega)\|_{\dot{H}^{1/2}_{\nu}}$ and $\|\Omega\|_{\dot{H}^{1/2}_{\nu}}$. First,
	$$ \|\Omega\|_{\dot{H}^{1/2}_{\nu}} \leq A_{\mu} \|\mathcal{D}(\Omega)\|_{\dot{H}^{1/2}_{\nu}} + 2A_{\rho}\|f\|_{\dot{H}^{1/2}_{\nu}}.$$
	Hence, we need to bound $\|\mathcal{D}(\Omega)\|_{\dot{H}^{1/2}_{\nu}} $ appropriately:
	\begin{align*}
	\|\mathcal{D}(\Omega)\|_{\dot{H}^{1/2}_{\nu}} &\leq 2\sum_{n\geq 0}  \||\xi|^{\frac{1}{2}}e^{\nu t |\xi|} (\ast^{2n+1}|\cdot||\hat{f}(\cdot)|)\ast |\widehat{\Omega}(\cdot)|\|_{L^{2}_{\nu}}\\
	&\leq 2\sum_{n\geq 0}  \|f\|_{\dot{\mathcal{F}}^{1,1}_{\nu}}^{2n+1}\|\Omega\|_{\dot{H}^{1/2}_{\nu}} + 2(2n+1)\|f\|_{\dot{\mathcal{F}}^{3/2,1}_{\nu}}\|f\|_{\dot{\mathcal{F}}^{1,1}_{\nu}}^{2n}\|\Omega\|_{L^{2}_{\nu}}.
	\end{align*}
	Using this estimate for $\mathcal{D}(\Omega)$,
	\begin{multline*}\|\Omega\|_{\dot{H}^{1/2}_{\nu}} \leq(1-2A_{\mu} \sum_{n\geq 0}  \|f\|_{\dot{\mathcal{F}}^{1,1}_{\nu}}^{2n+1})^{-1}\\\cdot \Big(2A_{\mu} \sum_{n\geq 0} (2n+1)\|f\|_{\dot{\mathcal{F}}^{3/2,1}_{\nu}}\|f\|_{\dot{\mathcal{F}}^{1,1}_{\nu}}^{2n}\|\Omega\|_{L^{2}_{\nu}} + 2A_{\rho}\|f\|_{\dot{H}^{1/2}_{\nu}}\Big).
	\end{multline*}
	For $\|f\|_{\dot{\mathcal{F}}^{1,1}_{\nu}}$ of our medium size, the inverted term on the right hand side above is a bounded constant. Also,
	\begin{multline*}
	\|\mathcal{D}(\Omega)\|_{\dot{H}^{1/2}_{\nu}} \leq \sum_{n\geq 0}  \|f\|_{\dot{\mathcal{F}}^{1,1}_{\nu}}^{2n+1}(1-2A_{\mu} \sum_{n\geq 0}  \|f\|_{\dot{\mathcal{F}}^{1,1}_{\nu}}^{2n+1})^{-1}\\\cdot \Big(2A_{\mu} \sum_{n\geq 0} (2n+1)\|f\|_{\dot{\mathcal{F}}^{3/2,1}_{\nu}}\|f\|_{\dot{\mathcal{F}}^{1,1}_{\nu}}^{2n}\|\Omega\|_{L^{2}_{\nu}} + 2A_{\rho}\|f\|_{\dot{H}^{1/2}_{\nu}}\Big)\\ + (2n+1)\|f\|_{\dot{\mathcal{F}}^{3/2,1}_{\nu}}\|f\|_{\dot{\mathcal{F}}^{1,1}_{\nu}}^{2n}\|\Omega\|_{L^{2}_{\nu}}\\
	\leq C(\|f\|_{\dot{\mathcal{F}}^{1,1}}) \Big(\|f\|_{\dot{\mathcal{F}}^{3/2,1}_{\nu}}\|\Omega\|_{L^{2}_{\nu}} + \|f\|_{\dot{H}^{1/2}_{\nu}}\Big).
	\end{multline*}
	Now, it can be seen that $\|\Omega\|_{L^{2}_{\nu}} \leq \tilde{C}(\|f\|_{\dot{\mathcal{F}}^{1,1}_{\nu}})\|f\|_{L^{2}_{\nu}}$ where $\tilde{C}(\|f\|_{\dot{\mathcal{F}}^{1,1}}) \rightarrow 0$ as $\|f\|_{\dot{\mathcal{F}}^{1,1}_{\nu}}\rightarrow 0$.
	Thus, summarizing, we can pick $\epsilon > 0$ small enough in the Young's inequality step in the bounds of the integral terms of $N_{2}$ and the other nonlinear terms,
	\begin{multline*}
	\frac{1}{2}\frac{d}{dt}\|f\|_{L^{2}_{\nu}}^{2}(t)  \leq  \Big(\nu-A_{\rho} + c(\epsilon,\|f\|_{\dot{\mathcal{F}}^{1,1}_{\nu}})\Big)\|f\|_{\dot{H}^{1/2}_{\nu}}^{2} + \frac{1}{2\epsilon}\tilde{c}(\|f\|_{\dot{\mathcal{F}}^{1,1}_{\nu}})\|f\|_{\dot{\mathcal{F}}^{3/2,1}_{\nu}}^{2}\|f\|_{L^{2}_{\nu}}^{2},
	\end{multline*}
	where $c(\epsilon,\|f\|_{\dot{\mathcal{F}}^{1,1}}) \rightarrow 0$ as $\|f\|_{\dot{\mathcal{F}}^{1,1}_{\nu}}\rightarrow 0$ or as $\epsilon\rightarrow 0$ and $\tilde{c}(\|f\|_{\dot{\mathcal{F}}^{1,1}}) \rightarrow 0$ as $\|f\|_{\dot{\mathcal{F}}^{1,1}_{\nu}}\rightarrow 0$. Hence, picking $\epsilon$ sufficiently small, but not $0$, the first term on the right hand side is negative. By Gronwall's inequality, we obtain
	$$\|f\|_{L^{2}_{\nu}}^2(t) \leq \|f_{0}\|_{L^{2}}^2\exp\Big(\frac{1}{2\epsilon}\tilde{c}(\|f\|_{\dot{\mathcal{F}}^{1,1}_{\nu}}(t))\int_{0}^{t}\|f\|_{\dot{\mathcal{F}}^{3/2,1}_{\nu}}^{2}(\tau) d\tau\Big).$$
	Finally, the exponential term on the right hand side is uniformly bounded because by interpolation
	$$\int_{0}^{t}\!\|f\|_{\dot{\mathcal{F}}^{3/2,1}_{\nu}}^{2}(\tau) d\tau \leq \int_{0}^{t}\!\|f\|_{\dot{\mathcal{F}}^{1,1}_{\nu}}(\tau) \|f\|_{\dot{\mathcal{F}}^{2,1}_{\nu}}(\tau) d\tau \leq  \|f_0\|_{\dot{\mathcal{F}}^{1,1}}\int_{0}^{t} \!\|f\|_{\dot{\mathcal{F}}^{2,1}_{\nu}}(\tau) d\tau
	\leq \|f_{0}\|_{\dot{\mathcal{F}}^{1,1}}^{2}.$$
This completes the proof.
\end{proof}
Next, recall the notation
$$\|f\|_{\dot{H}^{s}_{\nu}}=\|f\|_{\dot{\mathcal{F}}^{s,2}_{\nu}} =  \int |\xi|^{2s} e^{2|\xi| t \nu} |\hat{f}(\xi)|^{2} d\xi.$$
We will use the following inequality on the time derivative of the $H^{s}_{\nu}$ norm when performing decay estimates in $L^{2}$ spaces:
\begin{prop}\label{Hsdecreaseprop}
	Let $1/2\leq s\leq 3/2$ and assume $f_{0}\in \dot{\mathcal{F}}^{1,1}\cap L^{2}$ satisfying \eqref{condition}. Then,
	\begin{equation}\label{Hsdecrease}
	\frac{d}{dt} \|f\|_{\dot{H}^{s}_{\nu}} \leq -C \|f\|_{\dot{H}^{s+1/2}_{\nu}}.
	\end{equation}
\end{prop}

\begin{proof}
	Differentiating the quantity $\|f\|_{\dot{H}^{s}_{\nu}}$ and integrating by parts we obtain
	\begin{align*}
	\frac{1}{2}\frac{d}{dt} \|f\|_{\dot{H}^{s}_{\nu}}^{2} = \nu \|f\|_{\dot{H}^{s+1/2}_{\nu}}^{2} - A_{\rho} \|f\|_{\dot{H}^{s+1/2}_{\nu}}^{2} + K_{1} + K_{2} + K_{3},
	\end{align*}
	where the terms $K_i$ corresponds to the nonlinear terms $N_i$ in \eqref{decompnonlinear}. Then we have that
	$$ K_{1} \leq \frac{A_\mu}{2}\int |\xi|^{2s}e^{2 \nu t |\xi|} |\hat{f}(\xi)| |\Lambda \mathcal{D}(\Omega)(\xi)| d\xi.$$
	Using the identity $\Lambda = R_{1}\partial_{\alpha_{1}} + R_{2}\partial_{\alpha_{2}}$, it suffices to prove the following bounds on $\partial_{\alpha_{i}}\mathcal{D}(\Omega)$:
	\begin{align*}
	\int |\xi|^{2s} e^{2 \nu t |\xi|}|\hat{f}(\xi)| |R_{1}\partial_{\alpha_{1}}\mathcal{D}(\Omega)(\xi)| d\xi &\leq \int |\xi|^{2s} |\hat{f}(\xi)| |\partial_{\alpha_{1}}\mathcal{D}(\Omega)(\xi)| d\xi \\
	&\leq \|f\|_{\dot{H}^{s+1/2}_{\nu}}\|\partial_{\alpha_{1}}\mathcal{D}(\Omega)\|_{\dot{H}^{s-1/2}_{\nu}}.
	\end{align*}
	Hence, it suffices to appropriately bound $\|\partial_{\alpha_{1}}\mathcal{D}(\Omega)\|_{\dot{H}^{s-1/2}_{\nu}}$. Using \eqref{partialD} we have that
	\begin{multline*}
	\|\partial_{\alpha_{1}}\mathcal{D}(\Omega)\|_{\dot{H}^{s-1/2}_{\nu}} \leq 2\|BR_{1}\|_{\dot{H}^{s-1/2}_{\nu}} + 2\|BR_{3}\partial_{\alpha_{1}}f\|_{\dot{H}^{s-1/2}_{\nu}}\\
	\leq 2(\|BR_{1}\|_{\dot{H}^{s-1/2}_{\nu}} + \|BR_{3}\|_{\dot{H}^{s-1/2}_{\nu}}\|f\|_{\dot{\mathcal{F}}^{1,1}_{\nu}}\\ + \|f\|_{\dot{H}^{s+1/2}_{\nu}}\|BR_{3}\|_{\mathcal{F}^{0,1}_{\nu}}).
	\end{multline*}
	Similarly to previous estimates in Section \ref{secw}, we use the triangle inequality and Young's inequality to obtain that
	\begin{multline*}
	\|BR_{11}\|_{\dot{H}^{s-1/2}_{\nu}} \leq \frac{1}{2} \Bigg( \sum_{n\geq 0} \Big\||\xi|^{s-1/2}e^{\nu t |\xi|}\Big(|\widehat{\omega_{3}}(\cdot)|\ast (\ast^{2n} |\cdot||\hat{f}(\cdot)| )\Big)(\xi)\Big\|_{L^{2}_{\xi}} \Bigg) \\
	\leq \frac{1}{2} \Bigg( \sum_{n\geq 1} 2n  \Big\|(e^{\nu t |\cdot|}|\widehat{\omega_{3}}(\cdot)|)\ast (\ast^{2n-1} |\cdot|e^{\nu t |\cdot|}|\hat{f}(\cdot)|) \ast |\cdot|^{s+1/2}e^{\nu t |\cdot|}|\hat{f}(\cdot)| \Big\|_{L^{2}} \Bigg) \\+\frac{1}{2} \Bigg( \sum_{n\geq 0}  \Big\|(|\cdot|^{s-1/2}|e^{\nu t |\cdot|}\widehat{\omega_{3}}(\cdot)|)\ast (\ast^{2n} |\cdot|e^{\nu t |\cdot|}|\hat{f}(\cdot)|) \Big\|_{L^{2}} \Bigg)  \\
	\leq \frac{1}{2} \sum_{n\geq 1} 2n  \|\omega_{3}\|_{\mathcal{F}^{0,1}_{\nu}}\|f\|_{\dot{\mathcal{F}}^{1,1}_{\nu}}^{2n-1}\|f\|_{\dot{H}^{s+1/2}_{\nu}} +  \frac{1}{2} \sum_{n\geq 0}  \|\omega_{3}\|_{\dot{H}^{s-1/2}_{\nu}}\|f\|_{\dot{\mathcal{F}}^{1,1}_{\nu}}^{2n},
	\end{multline*}
	and
	\begin{multline*}
	\|BR_{12}\|_{\dot{H}^{s-1/2}_{\nu}} \leq \frac{1}{2} \Bigg( \sum_{n\geq 0} \Big\||\xi|^{s-1/2}e^{\nu t |\xi|}\Big(|\widehat{\omega_{2}}(\cdot)|\ast (\ast^{2n+1} |\cdot||\hat{f}(\cdot)| )\Big)(\xi)\Big\|_{L^{2}_{\xi}} \Bigg) \\
	\leq \frac{1}{2} \Bigg( \sum_{n\geq 0} (2n+1)  \Big\|(e^{\nu t |\cdot|}|\widehat{\omega_{2}}(\cdot)|)\ast (\ast^{2n} |\cdot|e^{\nu t |\cdot|}|\hat{f}(\cdot)|) \ast |\cdot|^{s+1/2}e^{\nu t |\cdot|}|\hat{f}(\cdot)| \Big\|_{L^{2}} \Bigg) \\+\frac{1}{2} \Bigg( \sum_{n\geq 0} \Big\|(|\cdot|^{s-1/2}|e^{\nu t |\cdot|}\widehat{\omega_{2}}(\cdot)|)\ast (\ast^{2n+1} |\cdot|e^{\nu t |\cdot|}|\hat{f}(\cdot)|) \Big\|_{L^{2}} \Bigg)  \\
	\leq \frac{1}{2} \sum_{n\geq 0} (2n+1)  \|\omega_{2}\|_{\mathcal{F}^{0,1}_{\nu}}\|f\|_{\dot{\mathcal{F}}^{1,1}_{\nu}}^{2n}\|f\|_{\dot{H}^{s+1/2}_{\nu}} +  \frac{1}{2} \sum_{n\geq 0} \|\omega_{2}\|_{\dot{H}^{s-1/2}_{\nu}}\|f\|_{\dot{\mathcal{F}}^{1,1}_{\nu}}^{2n+1},
	\end{multline*}
	and
	\begin{multline*}
	\|BR_{3}\|_{\dot{H}^{s-1/2}_{\nu}} \leq \frac{1}{2} \Bigg( \sum_{n\geq 0}  \Big\||\xi|^{s-1/2}e^{\nu t |\xi|}\Big((|\widehat{\omega_{1}}(\cdot)|+|\widehat{\omega_{2}}(\cdot)|)\ast (\ast^{2n} |\cdot||\hat{f}(\cdot)| )\Big)(\xi)\Big\|_{L^{2}_{\xi}} \Bigg) \\
	\leq \frac{1}{2} \Bigg( \sum_{n\geq 1} 2n  \Big\|(e^{\nu t |\cdot|}|\widehat{\omega_{1}}(\cdot)|+e^{\nu t |\cdot|}|\widehat{\omega_{2}}(\cdot)|)\ast (\ast^{2n-1} |\cdot|e^{\nu t |\cdot|}|\hat{f}(\cdot)|) \ast |\cdot|^{s+1/2}e^{\nu t |\cdot|}|\hat{f}(\cdot)| \Big\|_{L^{2}} \\+ \sum_{n\geq 0}  \Big\|(|\cdot|^{s-1/2}|e^{\nu t |\cdot|}\widehat{\omega_{1}}(\cdot)|+|\cdot|^{s-1/2}|e^{\nu t |\cdot|}\widehat{\omega_{2}}(\cdot)|)\ast (\ast^{2n} |\cdot|e^{\nu t |\cdot|}|\hat{f}(\cdot)|) \Big\|_{L^{2}} \Bigg)  \\
	\leq \frac{1}{2} \sum_{n\geq 1} 2n  (\|\omega_{1}\|_{\mathcal{F}^{0,1}_{\nu}}+\|\omega_{2}\|_{\mathcal{F}^{0,1}_{\nu}})\|f\|_{\dot{\mathcal{F}}^{1,1}_{\nu}}^{2n-1}\|f\|_{\dot{H}^{s+1/2}_{\nu}} \\+  \frac{1}{2} \sum_{n\geq 0} (\|\omega_{1}\|_{\dot{H}^{s-1/2}_{\nu}}+\|\omega_{2}\|_{\dot{H}^{s-1/2}_{\nu}})\|f\|_{\dot{\mathcal{F}}^{1,1}_{\nu}}^{2n}.
	\end{multline*}
	Hence, we now have to prove estimates on $\|\omega_{i}\|_{\dot{H}^{s-1/2}_{\nu}}$ for $1/2\leq s \leq 3/2$. This follows similar patterns:
	\begin{equation*}
	\|\partial_{\alpha_{1}}\Omega\|_{\dot{H}^{s-1/2}_{\nu}} \leq  2A_{\rho}\|f\|_{\dot{H}^{s+1/2}_{\nu}} + 2A_{\mu}\|BR_{1}\|_{\dot{H}^{s-1/2}_{\nu}} + 2A_{\mu}\|BR_{3}\partial_{\alpha_{1}}f\|_{\dot{H}^{s-1/2}_{\nu}}.
	\end{equation*}
	Notice that using the triangle inequality as above on $|\xi|^{s-1/2}$, since $0\leq s-1/2\leq 1$, we obtain analogously to the steps in Section \ref{secw} that
	\begin{equation*}
	\|\partial_{\alpha_{i}}\Omega\|_{\dot{H}^{s-1/2}_{\nu}} \leq 2A_\rho C_{4,\nu}\|f\|_{\dot{H}^{s+1/2}_{\nu}}.
	\end{equation*}
	Moreover, for $i=1,2$
	\begin{equation*}
	\|\omega_{i}\|_{\dot{H}^{s-1/2}_{\nu}} \leq 2A_\rho C_{4,\nu}\|f\|_{\dot{H}^{s+1/2}_{\nu}},
	\end{equation*}
	and
	\begin{equation*}
	\|\omega_{3}\|_{\dot{H}^{s-1/2}_{\nu}} \leq 4A_\mu A_\rho \|f\|_{\dot{\mathcal{F}}^{1,1}}^2(C_{5,\nu}  + 3C_{2,\nu}) \|f\|_{\dot{H}^{s+1/2}_{\nu}}.
	\end{equation*}
	Now we can follow the steps in Proposition \ref{normdecreaseprop}. Plugging in the estimates above and performing similar estimates for $K_{2}$ and $K_{3}$, we obtain for $1/2\leq s \leq 3/2$
	\begin{equation}
	\frac{d}{dt} \|f\|_{\dot{H}^{s}_{\nu}} \leq -C \|f\|_{\dot{H}^{s+1/2}_{\nu}},
	\end{equation}
	for a positive constant $C$ depending on $f_{0}$ and $\nu$.
\end{proof}

\section{Large-Time Decay of Analytic Norms}\label{DecaySection}
In this section, we begin by proving the Decay Lemma we will use to show large time decay of solutions to the Muskat problem:
\begin{lemma}[Decay Lemma]\label{decaylemma}
	Suppose $\|g\|_{\dot{\mathcal{F}}^{s_{1},p}_{\nu}}^{p} (t) \leq C_{0}$ and
	\bea\label{decreaseineq}
	\frac{d}{d{t}}\|g\|_{\dot{\mathcal{F}}^{s_{2},p}_{\nu}}^{p}(t) \leq -C\|g\|_{\dot{\mathcal{F}}^{s_{2}+1/p,p}_{\nu}}^{p} (t)
	\eea
	such that $s_{1} \leq s_{2}$ and $p\in [1,\infty)$.
	Then
	$$\|g\|_{\dot{\mathcal{F}}^{s_{2},p}_{\nu}}^{p}(t) \lesssim (1+t)^{(s_{1}-s_{2})p} .$$
\end{lemma}
\begin{proof}
	Consider $r > 0$. Then
	\begin{align*}
	\|g\|_{\dot{\mathcal{F}}^{r,p}_{\nu}}^{p} &= \int e^{\nu t p|\xi|} |\xi|^{rp}|\hat{g}(\xi)|^{p} d\xi \\
	&\geq \int_{|\xi|> (1+\delta t)^{s}} e^{\nu t p|\xi|} |\xi|^{rp}|\hat{g}(\xi)|^{p} d\xi \\
	&\geq (1+\delta t)^{s} \int_{|\xi|> (1+\delta t)^{s}} e^{\nu t p|\xi|} |\xi|^{(r-1/p)p}|\hat{g}(\xi)|^{p} d\xi\\
	&= (1+\delta t)^{s}\Big(\|g\|_{\dot{\mathcal{F}}^{r-1/p,p}_{\nu}}^p-\int_{|\xi| \leq (1+\delta t)^{s}} e^{\nu t p|\xi|} |\xi|^{(r-1/p)p}|\hat{g}(\xi)|^{p} d\xi\Big).
	\end{align*}
	We can use \eqref{decreaseineq} and the above argument with $r = s_{2}+1/p$ to obtain that
	\begin{multline*}
	\frac{d}{d{t}}\|g\|_{\dot{\mathcal{F}}^{s_{2},p}_{\nu}}^{p} + C(1+\delta t)^{s}\|g\|_{\dot{\mathcal{F}}^{s_{2},p}_{\nu}}^{p}
	\leq -C\|g\|_{\dot{\mathcal{F}}^{s_{2}+1/p,p}_{\nu}}^{p} + C(1+\delta t)^{s}\|g\|_{\dot{\mathcal{F}}^{s_{2},p}_{\nu}}^{p}\\
	\leq C(1+\delta t)^{s}\Big(\int_{|\xi| \leq (1+\delta t)^{s}} e^{\nu t p|\xi|} |\xi|^{s_{2}p}|\hat{g}(\xi)|^{p} d\xi\Big)\\
	\leq C(1+\delta t)^{s (s_{2}-s_{1}) p}(1+\delta t)^{s}\Big(\int_{|\xi| \leq (1+\delta t)^{s}} e^{\nu t p|\xi|} |\xi|^{s_{1}p}|\hat{g}(\xi)|^{p} d\xi\Big)\\
	\leq C(1+\delta t)^{s (s_{2}-s_{1}) p}(1+\delta t)^{s}\|g\|_{\dot{\mathcal{F}}^{s_{1},p}_{\nu}}^{p}\\
	\leq CC_{0}(1+\delta t)^{s (s_{2}-s_{1}) p}(1+\delta t)^{s}.
	\end{multline*}
	Now, let $\sigma > (s_{2}-s_{1})p$ and choose $\delta$ such that $\delta\sigma=C$, $s=-1$. Then
	\begin{align*}
	\frac{d}{d{t}}((1+\delta t)^{\sigma}\|g\|_{\dot{\mathcal{F}}^{s_{2},p}_{\nu}}^{p}) &= (1+\delta t)^{\sigma} \frac{d}{dt}\|g\|_{\dot{\mathcal{F}}^{s_{2},p}_{\nu}}^{p}+ \sigma\delta \|g\|_{\dot{\mathcal{F}}^{s_{2},p}_{\nu}}^{p}(1+\delta t)^{\sigma-1}\\
	&=  (1+\delta t)^{\sigma}\frac{d}{dt}\|g\|_{\dot{\mathcal{F}}^{s_{2},p}_{\nu}}^{p} + C \|g\|_{\dot{\mathcal{F}}^{s_{2},p}_{\nu}}(1+\delta t)^{\sigma-1}\\
	&=  (1+\delta t)^{\sigma}(\frac{d}{dt}\|g\|_{\dot{\mathcal{F}}^{s_{2},p}_{\nu}}^{p} + C \|g\|_{\dot{\mathcal{F}}^{s_{2},p}_{\nu}}^{p}(1+\delta t)^{-1})\\
	&\leq CC_{0}(1+\delta t)^{\sigma -(s_{2}-s_{1})p-1}.
	\end{align*}
	Integrating in time we obtain that
	$$ (1+\delta t)^{\sigma}\|g\|_{\dot{\mathcal{F}}^{s_{2},p}_{\nu}}^{p} \leq \frac{\tilde{C}}{\delta t} (1+ (1+\delta)^{\sigma-(s_2-s_1)p})$$
	for some constant $\tilde{C}$. Dividing both sides of the inequality by $(1+\delta t)^{\sigma}$ we obtain our results.
\end{proof}
We can now use this lemma to prove large-time decay rates for the analytic norms. By Holder's inequality, for $s> -d/2$ and $r > s+ d/2$
\begin{align*}
\|f\|_{\dot{\mathcal{F}}^{s,1}_{\nu}} &\leq \|f\|_{H^{r}_{\nu}}\Big\|\frac{|\xi|^{s}}{(1+|\xi|^{2})^{r/2}}\Big\|_{L^{1}}\\
&\lesssim \|f\|_{H^{r}_{\nu}}
\end{align*}
Hence, by the estimate \eqref{Hsdecrease}, we obtain for $-1<s'<0$:
\bea\label{fs1negative}
\|f\|_{\dot{\mathcal{F}}^{s',1}_{\nu}}(t) \leq C_{s}
\eea
for a fixed constant $C_{s}$.
By the Decay Lemma, this implies that
\bea\label{initialdecay}
\|f\|_{\dot{\mathcal{F}}^{s,1}_{\nu}} \lesssim (1+t)^{-1-s+\lambda}
\eea
for $0\leq s \leq 1$ and arbitrarily small $\lambda > 0$. This proves Theorem \ref{DecayTheorem}.
We can further demonstrate decay of other analytic norms. For example, the quantities $\|f\|_{\dot{\mathcal{F}}^{s,1}_{\nu}}$ for $s > 1$ also decay in time. First, note that
$$\|f\|_{\dot{\mathcal{F}}^{s,1}}(t) \leq \|e^{-\nu t |\xi|}|\xi|^{s-1}\|_{L^{\infty}_{\xi}}\|f\|_{\mathcal{F}^{0,1}_{\nu}} \leq \|e^{-\nu t |\xi|}|\xi|^{s-1}\|_{L^{\infty}_{\xi}}k(A_{\mu}) < \infty$$
for any $t > 0$. Moreover, using the weighted triangle inequality
$$|\xi|^{s} \leq (n+1)^{s}(|\xi-\xi_{1}|^{s} + |\xi_{1}-\xi_{2}|^{s}+ \cdots + |\xi_{n}|^{s}),$$ it can be seen that
$$\frac{d}{dt}\|f\|_{\dot{\mathcal{F}}^{s,1}_{\nu}}(t) \leq -C(s)\|f\|_{\dot{\mathcal{F}}^{s+1,1}_{\nu}}(t)$$
for a positive constant $C(s)$ depending on $s$ and $\|f\|_{\dot{\mathcal{F}}^{1,1}_{\nu}}$ when $\|f\|_{\dot{\mathcal{F}}^{1,1}_{\nu}}$ is sufficiently small. However, by \eqref{initialdecay} for $s=1$, the quantity $\|f\|_{\dot{\mathcal{F}}^{1,1}_{\nu}}(t)$ does indeed decay to a sufficiently small quantity when $t > T_{s}$ for some $T_{s} > 0$ depending on $s$ and the initial data $f_{0}$. Hence, since, as noted earlier, $\|f\|_{\dot{\mathcal{F}}^{s,1}_{\nu}}(T_{s}) < \infty$, we can apply the Decay Lemma to obtain
\begin{thm}\label{fs1decay}
	Let $s\geq 1$. Then, $\|f\|_{\dot{\mathcal{F}}^{s,1}}(t) < \infty$ for all $t > 0$ and we have the decay
	$$\|f\|_{\dot{\mathcal{F}}^{s,1}_{\nu}} \leq C_{1,s}t^{-s-1}$$
	for $t > T_{s}$ and $C_{1,s}$ depending on $s$ and the initial data $f_{0}$.
\end{thm}
We can similarly see by the uniform bounds of Theorem \ref{L2}, that $\|f\|_{\dot{H}^{s}}(t) < \infty$ for all $s\geq0$ and $t>0$. Hence, by similar arguments to above and Proposition \ref{Hsdecreaseprop}, the Decay Lemma \ref{decaylemma} with $p=2$,
\begin{thm}\label{L2decay}
	Suppose $s\geq 1/2$. Then $\|f\|_{\dot{H}^{s}} < \infty$ for all $t> 0$ and we have the decay
	$$\|f\|_{\dot{H}^{s}_{\nu}}^{2} \leq C_{2,s}t^{-2s}$$
	for $t > T_{s}$ and $C_{2,s}$ depending on $s$ and the initial data $f_{0}$.
\end{thm}

\section{Ill-posedness}\label{illposedsection}

In this section we show that the Muskat problem in the unstable case $\rho_{1} > \rho_{2}$ is ill-posed for any Sobolev space $H^s$ with $s>0$.
First we notice that our initial data $f\in L^2\cap \dot{\mathcal{F}}^{1,1}$ need not be in $H^s$.
\begin{lemma} There exists a function $f\in L^2\cap \dot{\mathcal{F}}^{1,1}$ with $$\|f_{0}\|_{L^{2}} < \infty,  \quad \|f_{0}\|_{\dot{\mathcal{F}}^{1,1}} < k_{\mu}$$ 
for a constant $k_{\mu}$ of medium size such that $f\notin H^s$ for any $s>0$.
\end{lemma}
\begin{proof} We give an explicit counterexample. Consider a radial function $f:\mathbb{R}^{2}\rightarrow \mathbb{R}$. Let for $n\geq N$ for some $N > 0$ integer
\[   
|\xi \hat{f}(\xi)| = r |\hat{f}(r)|= 
\begin{cases}
n^{\sigma} &\quad\text{if   } r\in [n^{\delta}, n^{\delta}+1/n^{\gamma}] \\
0 &\quad\text{otherwise},\\
\end{cases}
\]
where $\sigma$, $\delta$ and $\gamma$ are positive. Then one can compute
\begin{align*}
\|f\|_{\dot{\mathcal{F}}^{1,1}} &= 2\pi \int_{0}^{\infty} r^{2}|\hat{f}(r)|dr\leq 2\pi \sum_{n\geq N} n^{\sigma + \delta -\gamma} + n^{\sigma -2\gamma} < \infty
\end{align*}
when $\sigma + \delta -\gamma < -1$.
For $0\leq s < 1/2$ we have that 
$$\|f\|_{\dot{H}^{s}}^{2} = \int_{0}^{\infty} (r|\hat{f}(r)|)^{2}r^{2s-1}dr$$ and having chosen $N >0$ appropriately large,
\begin{align*}
2^{2s-1}2\pi\sum_{n\geq N} n^{2\sigma+\delta(2s-1)-\gamma} &\leq 2\pi\sum_{n\geq N} n^{2\sigma-\gamma}(n^{\delta}+n^{-\gamma})^{2s-1}
\\&\leq 2\pi \int_{0}^{\infty} (r|\hat{f}(r)|)^{2}r^{2s-1}dr \\
&\leq 2\pi \sum_{n\geq N} n^{2\sigma + \delta(2s-1) -\gamma}.
\end{align*}
Hence, pick $\sigma$, $\delta$ and $\gamma$ such that 
$2\sigma + \delta(2s-1) -\gamma = -1$. Then, $\|f\|_{L^{2}} < +\infty$ and for $s > 0$, $\|f\|_{\dot{H}^{s}} = +\infty$. This counterexample gives the proof in the 3D case, as we can force $\|f\|_{\dot{\mathcal{F}}^{1,1}} < k_{\mu}$ by multiplying this counterexample by the appropriate constant.

In the 1D interface case, let for $n\geq N$ for some $N > 0$ integer
\[   
\xi \hat{f}(\xi) = 
\begin{cases}
n^{\sigma} &\quad\text{if   } \xi\in [n^{\delta}, n^{\delta}+1/n^{\gamma}] \\
0 &\quad\text{otherwise}\\
\end{cases}
\]
such that $\gamma > \sigma +1$, $2\delta + \gamma > 2\sigma + 1$ but $2\delta(1-s) + \gamma = 2\sigma + 1$. Then one can compute that $$\|f\|_{L^{2}} < \infty,  \|f\|_{\dot{\mathcal{F}}^{1,1}} < k_{\mu} \hspace{0.2cm}\textup{and}\hspace{0.2cm}\|f\|_{H^s}=+\infty.$$
\end{proof}
\begin{remark}
	This example can be adapted to show that even if $f\in\dot{\mathcal{F}}^{1,1}_{\nu}\cap L^{2}$, it need not be in $H^{s}$.
\end{remark}

\begin{thm}
	For every $\epsilon > 0$, there exists a solution $\tilde{f}$ to the unstable regime and $0< \delta < \epsilon$ such that $\|\tilde{f}\|_{H^{s}}(0) = \delta$ but $\|\tilde{f}\|_{H^{s}}(\delta) = \infty$.
\end{thm}
\begin{proof}
	Take $f_{0} \in L^{2}\cap \dot{\mathcal{F}}^{1,1}$ satisfying condition \eqref{condition} for the Muskat problem in the stable regime such that $\|f_{0}\|_{H^{s}} = \infty$.   By the gain of regularity in \eqref{gainofL2}
	$$\|f\|_{H^{s}}(\delta) \leq \|e^{-\nu\delta|\xi|}|\xi|^{s}\|_{L^{\infty}}\|f\|_{L^{2}_{\nu}}(\delta) \leq c(\delta)\|f_{0}\|_{L^{2}}\exp\Big(R(\|f_{0}\|_{\dot{\mathcal{F}}^{1,1}})\Big) < \epsilon $$
	by picking initial data with $\|f_{0}\|_{L^{2}}$ sufficiently small.
	If $f(x,t)$ is a solution to the stable case problem, then $\tilde{f}(x,t) = f(x,-t+\delta)$ is a solution to the unstable case $\rho_{1} > \rho_{2}$. Hence,
	$$\|\tilde{f}\|_{H^{s}}(0) = \|f\|_{H^{s}}(\delta) < \epsilon \text{   and  } \|\tilde{f}\|_{H^{s}}(\delta) = \|f\|_{H^{s}}(0) = \infty.$$
	This completes the proof.
\end{proof}

\section{Acknowledgements}

FG and EGJ were partially supported by the project P12-FQM-2466 of Junta de Andaluc\'ia, Spain, 
and the grant MTM2014-59488-P (Spain). FG, EGJ and NP were partially supported by the ERC through the Starting Grant project H2020-EU.1.1.-639227. EGJ was partially supported by MECD through a FPU grant from the Spanish Government. RMS was partially supported by the NSF grant DMS-1200747 (USA).

\end{document}